\documentclass[openright,a4paper,american,11pt]{amsart}

\usepackage[utf8]{inputenc}

\usepackage{amsmath}
\usepackage{amssymb}

\usepackage{babel}

\usepackage{amsfonts}
\input xy
\xyoption{all}
\usepackage[dvips]{graphicx}
\usepackage{boxedminipage}
\usepackage{color}

\newcommand{\ZZ}{{\mathbb Z}}

\newcommand{\CC}{{\mathbb C}}
\newcommand{\RR}{{\mathbb R}}

\newcommand{\FF}{{\mathbb F}}
\newcommand{\TT}{{\mathbb T}}

\newcommand{\D}{{\mathcal D}}
\newcommand{\C}{{\mathcal C}}
\newcommand{\E}{{\mathcal E}}
\newcommand{\F}{{\mathcal F}}

\newtheorem{thm}{Theorem}[section]
\newtheorem{defi}[thm]{Definition}
\newtheorem{prop}[thm]{Proposition}
\newtheorem{lemma}[thm]{Lemma}
\newtheorem{cor}[thm]{Corollary}

          {\theoremstyle{definition}
\newtheorem{rem}[thm]{Remark}}
          {\theoremstyle{definition}
\newtheorem{exa}[thm]{Example}}

\begin{document}
\title{Floor decompositions of tropical curves : the planar case}
\date{\today}
\author{Erwan Brugallé}
\address{Université Pierre et Marie Curie,  Paris 6, 175 rue du Chevaleret, 75 013 Paris, France}
\email{brugalle@math.jussieu.fr}
\author{Grigory Mikhalkin}
\address{Université de Genève, 2-4 rue du Lièvre, Case postale 64,
  1211 Genève, Suisse}
\email{grigory.mikhalkin@unige.ch}


\begin{abstract}
In \cite{Br7} we announced a formula to compute Gromov-Witten and
Welschinger invariants of some toric varieties, in terms of combinatorial
objects called floor diagrams.
We give here detailed  proofs in the tropical geometry framework, in the
case when the ambient variety is a complex surface, and give some
examples of computations using floor diagrams.
The focusing on dimension 2 is motivated by the special combinatoric
of floor diagrams compared to arbitrary dimension.

We treat a general toric surface case in this dimension:
the curve is given by an arbitrary lattice polygon and
include computation of Welschinger invariants with pairs
of conjugate points. See also \cite{FM} for combinatorial
treatment of floor diagrams in the projective case.
\end{abstract}

\maketitle

\section{Introduction}

Let $\Delta$ be a lattice polygon in $\RR^2$, $g$ a non-negative
integer, and $\omega$ a generic configuration of
$Card(\partial\Delta\cap\ZZ^2)-1+g$ points in $(\CC^*)^2$. Then, there exists a
finite number $N(\Delta,g)$ of complex algebraic curves in $(\CC^*)^2$
of genus $g$ and Newton polygon $\Delta$ passing through all
points in $\omega$. Moreover, $N(\Delta,g)$ doesn't depend on $\omega$
as long as it is generic. If the toric surface $Tor(\Delta)$
corresponding to $\Delta$ is
Fano, then the numbers $N(\Delta,g)$ are known as
\textit{Gromov-Witten invariants} of the surface
$Tor(\Delta)$. Kontsevich first computed in \cite{KonMan1} the series
$N(\Delta,0)$
for convex surfaces $Tor(\Delta)$, and Caporaso and Harris computed in
\cite{CapHar1}
all $N(\Delta,g)$'s where  $Tor(\Delta)$ is Fano or a Hirzebruch
surface.

Suppose now that the surface $Tor(\Delta)$ is equipped with a real
structure $conj$, i.e. $conj$ is a antiholomorphic involution on $Tor(\Delta)$. For
example, one can
take the tautological real structure given in $(\CC^*)^2$ by the
standard complex conjugation. Suppose moreover that $\omega$ is a real
configuration, i.e. $conj(\omega)=\omega$.  Then it is natural to
study the set $\RR\C(\omega)$ of real algebraic curves in $(\CC^*)^2$
of genus $g$ and Newton polygon $\Delta$ passing through all
points in $\omega$. It is not hard to see that, unlike in the
enumeration of complex curves, the cardinal of this set depends
heavily on $\omega$. However, Welschinger proved in \cite{Wel1}
that when $g=0$ and $Tor(\Delta)$ is Fano, one can define an
invariant. A real nodal curve $C$ in $Tor(\Delta)$ has two types
of real nodes, isolated ones (locally given by the equation
$x^2+y^2=0$) and non-isolated ones (locally given by the equation
$x^2-y^2=0$). Welschinger defined the \textit{mass} $m(C)$ of the curve $C$
as the number of isolated nodes of $C$, and proved that if $g=0$ and
$Tor(\Delta)$ is Fano, then the number
$$W(\Delta,r)=\sum_{C\in  \RR\C(\omega) }(-1)^{m(C)} $$
depends only on $\Delta$ and the number $r$ of pairs of complex
conjugated points in $\omega$.

 \textit{Tropical
  geometry} is an algebraic geometry over the tropical
semi-field $\TT=\RR\cup\{-\infty\}$ where the tropical addition is
taking the maximum, and the tropical multiplication is the classical
addition. As in the classical setting, given $\Delta$  a lattice
polygon in $\RR^2$, $g$ a non-negative
integer, and $\omega$ a generic configuration of
$Card(\partial\Delta\cap\ZZ^2)-1+g$ points in $(\RR^*)^2$, we can
enumerate
tropical curves in $\RR^2$ of genus $g$ and Newton polygon $\Delta$
passing through all points in $\omega$. It was proved in \cite{Mik1} that
provided that we count tropical curves with an appropriate
multiplicity, then the number of tropical curves does not depend on
$\omega$ and is equal to $N(\Delta,g)$. Moreover,  tropical
geometry allows also one to compute quite easily Welschinger
invariants $W(\Delta,r)$ of Fano toric surfaces equipped with the
tautological real structure (see \cite{Mik1}, \cite{Sh8}). This has
been the first systematic method
to compute Welschinger invariants of these surfaces.

In \cite{Br7}, we announced a formula to compute the numbers
$N(\Delta,g)$  and $W(\Delta,r)$ easily in terms of combinatorial
objects called
\textit{floor diagrams}. This diagrams encode  degeneracies of
tropical curves passing through some special configuration of
points. This paper is devoted to explain how floor diagrams can be
used to compute the numbers $N(\Delta,g)$  and $W(\Delta,r)$ (Theorems
\ref{NFD} and \ref{WFD}), and to
give some examples of concrete computations (section \ref{application}).
In \cite{Br7}, we announced a more general formula computing
Gromov-Witten and Welschinger invariants of some toric varieties of
any dimension. However, floor diagrams corresponding to plane curves
have a special combinatoric compared with the general case, and
deserve some special attention. Details of the proof of the general
formula given in \cite{Br7} will appear soon.

\vspace{2ex}

In section \ref{convention} we remind some convention we use throughout
this paper about graphs and lattice polygons. Then, we state in
section \ref{floor diagrams} our main
formulas computing the numbers  $N(\Delta,g)$  and $W(\Delta,r)$ when
$\Delta$ is a \textit{$h$-transverse} polygon.  We present tropical enumerative
geometry in section
 \ref{enum trop geo}, and prove our main formulas in section
 \ref{Floor trop}. We give some examples of computations using floor
 diagrams in section
\ref{application}, and we end this paper with some remarks in section
 \ref{further}.

\section{convention}\label{convention}

\subsection{Graphs}

In this paper, graphs are considered as (non necessarily compact) abstract
1 dimensional topological objects. Remind that a \textit{leaf} of a
graph is an edge which is non-compact or adjacent
 to a 1-valent
 vertex.
Given a graph $\Gamma$, we denote
by

\begin{itemize}
\item $\text{Vert}(\Gamma)$ the set of
its vertices,
\item  $\text{End}(\Gamma)$ the set of
its 1-valent vertices,
\item $\text{Edge}(\Gamma)$
 the set of  its edges,
\item $\text{Edge}^\infty(\Gamma)$
 the set of its non-compact leaves.
\end{itemize}

If in addition $\Gamma$ is oriented so that there are no oriented cycles,
then there exists a natural partial
ordering on $\Gamma$ : an
 element $a$ of $\Gamma$ is greater than another element  $b$ if there
 exists an oriented path from $b$ to $a$. In this case, we denote by
$\text{Edge}^{+\infty}(\Gamma)$
(resp. $\text{Edge}^{-\infty}(\Gamma)$)  the set of
edges  $e$ in
$\text{Edge}^{\infty}(\Gamma)$ such that no vertex of $\Gamma$ is
greater (resp. smaller) than a point of $e$.

We say that $\Gamma$ is a \textit{weighted graph}  if each edge of
$\Gamma$ is prescribed a natural weight, i.e.  we are given a function
$\omega:\text{Edge}(\Gamma)\rightarrow \mathbb N^*$. 
Weight and orientation
allow
one to define the \textit{divergence} at the vertices. Namely, for a
vertex $v\in\text{Vert}(\Gamma)$ we define the divergence
$\text{div}(v)$ to be the sum of the weights of all incoming edges
minus the sum of the weights of all outgoing edges.

\subsection{Lattice polygons} We remind that a \textit{primitive
  integer vector}, or shortly a primitive vector, is a vector
$(\alpha,\beta)$ in $\ZZ^2$ whose coordinates are relatively prime.
A \textit{lattice polygon} $\Delta$ is a convex polygon in $\mathbb
R^2$ whose vertices
are in $\ZZ^2$. For such a polygon, we define
$$\partial_l\Delta=\{p\in\partial \Delta \ | \  \forall t>0, \ p + (-t,0)\notin
\Delta\},$$
$$\partial_r \Delta=\{p\in\partial \Delta \ | \  \forall t>0, \ p + (t,0)\notin
\Delta\}. $$

A lattice polygon $\Delta$ is said to be \textit{$h$-transverse} if
 any primitive vector parallel to an edge of $\partial_l\Delta$ or
 $\partial_r\Delta$  is of
 the form $(\alpha,\pm 1)$ with $\alpha$ in $\ZZ$.

 If $e$ is a lattice segment in $\RR^2$, we define the \textit{integer
   length}
of $e$ by
$l(e)=Card(e\cap \ZZ^2)-1$.
If $\Delta$ is a $h$-transverse polygon,
we define its  \textit{left directions} (resp. \textit{right
  directions}), denoted by $d_l(\Delta)$ (resp. $d_r(\Delta)$), as
the unordered list that consists of the elements $\alpha$ repeated  $l(e)$ times for
all edge vectors $e=\pm l(e)(\alpha,-1)$ of $\partial_l\Delta$
(resp.  $\partial_r\Delta$).
 If $\Delta$ has a bottom (resp. top)
horizontal edge $e$ then we set $d_-(\Delta)=l(e)$
(resp. $d_+(\Delta)=l(e)$) and $d_-(\Delta)=0$
(resp. $d_+(\Delta)=0$) otherwise.

There is a natural 1-1 correspondence between quadruples
$(d_l,d_r,d_-,d_+)$ and $h$-transverse polygons $\Delta$
considered up to translation
as the polygon can be easily reconstructed from such a quadruple.

We have
\begin{equation}\label{eq1}
Card(d_l(\Delta))  = Card( d_r(\Delta))   =Card(\partial_l\Delta\cap\ZZ^2)-1 =
Card(\partial_r\Delta\cap\ZZ^2) -1
\end{equation}
and
\begin{equation}\label{eq2}
2Card(d_l(\Delta))  +d_-(\Delta) +d_+(\Delta)=
Card(\partial\Delta\cap \ZZ^2)  .
\end{equation}
We call the cardinality $Card(d_l(\Delta))$ the {\em height} of $h$-transversal
polygon $\Delta$.

\begin{exa}
Some $h$-transverse polygons are depicted in Figure \ref{ex poly}. By
abuse of notation, we write unordered lists within brackets $\{\}$.
\end{exa}

\begin{figure}[h]
\centering
\begin{tabular}{cccc}
\includegraphics[height=2cm, angle=0]{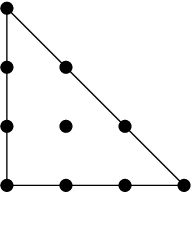}&
\includegraphics[height=2cm, angle=0]{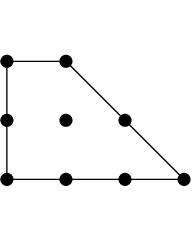}&
\includegraphics[height=2cm, angle=0]{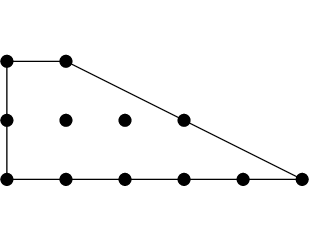}&
\includegraphics[height=2cm, angle=0]{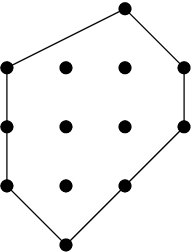}
\\$\begin{array}{ll} a)& d_l=\{0,0,0\}
\\&  d_r =\{1,1,1\}
\\ & d_-=3
\\& d_+=0
\end{array}$

 &$\begin{array}{ll} b)& d_l=\{0,0\}
 \\&  d_r=\{1,1\}
 \\ & d_-=3
 \\& d_+=1
 \end{array}$

&$\begin{array}{ll} c)& d_l=\{0,0\}
\\&  d_r=\{2,2\}
\\ & d_-=5
\\& d_+=1
\end{array}$

&$\begin{array}{ll} d)& d_l= \{1,0,0,-2\}
\\&  d_r=\{-1,-1,0,1\}
\\ & d_-=0
\\& d_+=0
\end{array}$

\end{tabular}
\caption{Examples of $h$-transverse polygons}
\label{ex poly}
\end{figure}

\begin{rem}
If $\Delta$ is a lattice polygon and if $v$ is a primitive
integer vector  such that for any edge $e$ of $\Delta$ we have
$|det(v,e)|\le l(e)$, then $\Delta$ is a $h$-transverse polygon after a
suitable change of coordinates in $SL_2(\ZZ)$.
\end{rem}

In this paper, we denote by $\Delta_d$ the lattice polygon with
vertices $(0,0)$, $(d,0)$, and $(0,d)$.

\section{Floor diagrams}\label{floor diagrams}
Here we define the combinatorial objects that can be used to replace
the algebraic curves in real and complex enumerative problems.
In this section, we fix an $h$-transverse lattice polygon $\Delta$.

\subsection{Enumeration of complex curves}

\begin{defi}\label{def fd}
A (plane) floor diagram
 $\D$  of genus $g$  and  Newton polygon
  $\Delta$ is the data of a  connected weighted oriented graph
$\Gamma$ and a map $\theta : \text{Vert}(\Gamma)\to \ZZ$ which satisfy
the
  following conditions

\begin{itemize}
\item the oriented graph $\Gamma$ is acyclic,
\item the first Betti number $b_1(\Gamma)$ equals $g$,
\item there are exactly $d_\pm(\Delta)$ edges in
  $\text{Edge}^{\pm\infty}(\Gamma)$, and all of them are of weight $1$,
\item the (unordered) collection of numbers $\theta(v)$, where
  $v$ goes through  vertices of $\Gamma$, coincides with $d_l(\Delta)$,
\item the (unordered) collection of numbers $\theta(v)+\text{div}(v)$, where
  $v$ goes through  vertices of $\Gamma$, coincides with $d_r(\Delta)$.
\end{itemize}
\end{defi}

In order to avoid too many notation, we will denote by the same letter
$\D$ a floor diagram and its underlying graph $\Gamma$.
Here are the convention we use to depict floor diagrams :
vertices of
$\D$
are represented by ellipses. We write $\theta(v)$ inside the ellipse
$v$ only
if $\theta(v)\ne 0$. Edges of $\D$
are represented by vertical lines, and
the orientation is
implicitly from down to up. We write the weight of an edge close to
it only if this weight is at least 2.
In the following,  we define $s=Card(\partial\Delta\cap
\ZZ^2) +g
-1$.

\begin{exa}
Figure \ref{ex FD} depicts an example of floor diagram for any
$h$-transverse polygon depicted in Figure \ref{ex poly}.
\end{exa}

\begin{figure}[h]
\centering
\begin{tabular}{ccccccc}
\includegraphics[height=3cm, angle=0]{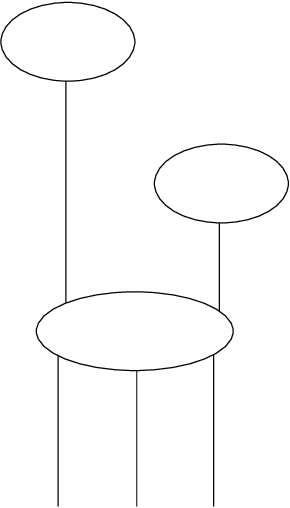}& \hspace{6ex}
&
\includegraphics[height=3cm, angle=0]{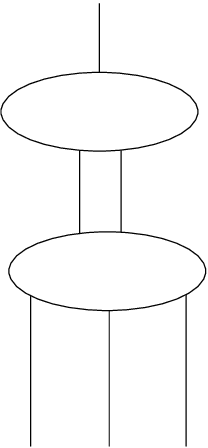}& \hspace{6ex}
&
\includegraphics[height=3cm, angle=0]{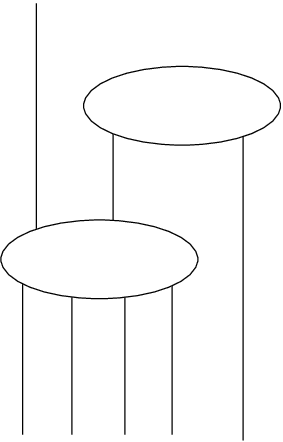}& \hspace{6ex}
&
\includegraphics[height=3cm, angle=0]{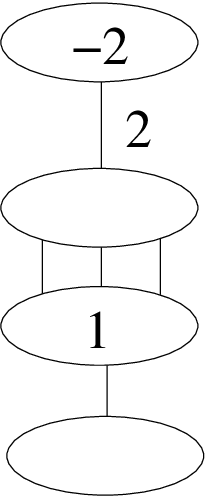}
\\
\\ a) $g=0$ && b) $g=1$ && c) $g=0$ && d) $g=2$
\end{tabular}
\caption{Examples of floor diagrams whose Newton polygon are depicted in
  Figure \ref{ex poly}}
\label{ex FD}
\end{figure}

Note that Equations (\ref{eq1}) and (\ref{eq2}) combined with Euler's
formula imply that for
any floor diagram  $\D$ of genus $g$ and Newton polygon $\Delta$ we have
$$Card(\text{Vert}(\D)) + Card(\text{Edge}(\D))=s.$$

A map $m$ between two partially ordered sets is said
\textit{increasing} if
$$m(i)>m(j)\Longrightarrow i>j$$

\begin{defi}\label{def marking}
A marking of a floor diagram $\D$ of genus $g$ and Newton polygon
$\Delta$ is an increasing  map $m  :
\{1,\dots,s\}\rightarrow
\mathcal{D}$ such that for any edge or vertex $x$ of $\D$, the set
$m^{-1}(x)$ consists of exactly one element.
\end{defi}

A floor diagram enhanced with a marking is called a \textit{marked floor
  diagram} and is said to be
marked by $m$.

\begin{defi}
Two marked floor diagrams $(\D,m)$ and $(\D',m')$ are called
\textit{equivalent} if there exists a homeomorphism of oriented graphs
$\phi : \D\to\D'$ such that $w=w'\circ \phi$, $\theta=\theta'\circ
\phi$, and $m=m'\circ \phi$.
\end{defi}

Hence, if $m(i)$ is an edge $e$ of $\D$, only the knowledge of $e$ is
important to
determine the equivalence class of $(\D,m)$, not
the position of $m(i)$ on $e$.
From now on, we consider marked floor diagrams up to equivalence.
To any (equivalence class of) marked floor diagram, we  assign a
sequence of non-negative
integers called \textit{multiplicities}~: a \textit{complex}
multiplicity, and some \textit{$r$-real} multiplicities.

\begin{defi}
The complex multiplicity of a marked floor diagram $\mathcal D$,
denoted by $\mu^\CC(\mathcal D)$, is defined as
$$\mu^\CC(\mathcal D)=
\prod_{e\in \text{Edge}(\mathcal D)}w(e)^2 $$
\end{defi}
Note that the complex multiplicity of a marked floor diagram
depends only on the underlying floor diagram. Next theorem is the
first of our two main formulas.

\begin{thm}\label{NFD}
For any $h$-transverse polygon $\Delta$ and any genus $g$, one has
$$N(\Delta,g)=\sum \mu^\CC(\mathcal D)$$
where the sum is taken over all marked floor diagrams of
genus $g$ and  Newton polygon $\Delta$.
\end{thm}

Theorem \ref{NFD} is a corollary of Proposition \ref{fd bij} proved in
section \ref{Floor trop}.

\begin{exa}
Using marked floor diagrams depicted in Figures \ref{comp cplx1} and
\ref{comp cplx2} we verify that

\vspace{2ex}

$N(\Delta_3,1)= 1$ (see Figure \ref{comp cplx1}a), $N(\Delta_3,0 )=12$
(see Figure \ref{comp cplx1}b,c,d).
\vspace{2ex}

$N(\Delta,0)= 84$ (see Figure \ref{comp cplx2}),  where $\Delta$
  is the polygon depicted in Figure \ref{ex poly}c.
   \end{exa}

\begin{figure}[h]
\centering
\begin{tabular}{cccc}
\includegraphics[height=3cm, angle=0]{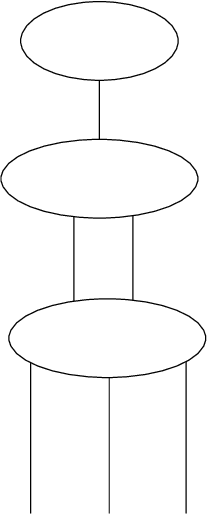}&
\includegraphics[height=3cm, angle=0]{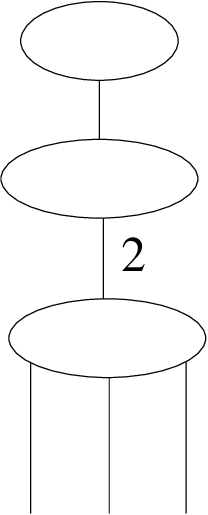}&
\includegraphics[height=3cm, angle=0]{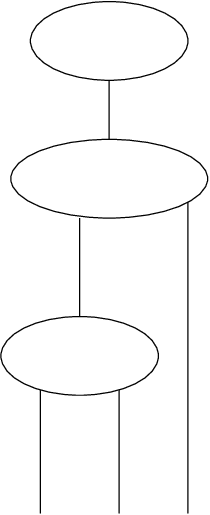}&
\includegraphics[height=3cm, angle=0]{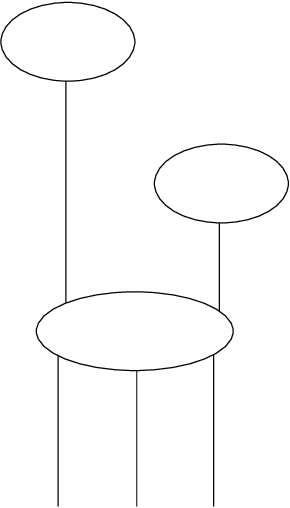}
\\
\\ a) $\mu^\CC=1$, 1 marking & b)  $\mu^\CC=4$, 1 marking & c)
$\mu^\CC=1$, 5 markings  & d)  $\mu^\CC=1$, 3 markings
\end{tabular}
\caption{Floor diagrams of genus 1 and 0, and  Newton polygon $\Delta_3$}
\label{comp cplx1}
\end{figure}

\begin{figure}[h]
\centering
\begin{tabular}{cccc}
\includegraphics[height=3cm, angle=0]{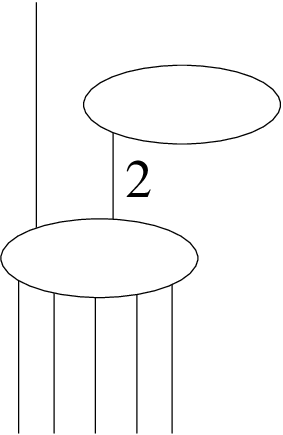}&
\includegraphics[height=3cm, angle=0]{Figures/ExFDc.eps}&
\includegraphics[height=3cm, angle=0]{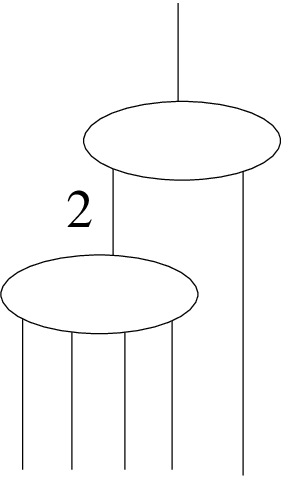}&
\includegraphics[height=3cm, angle=0]{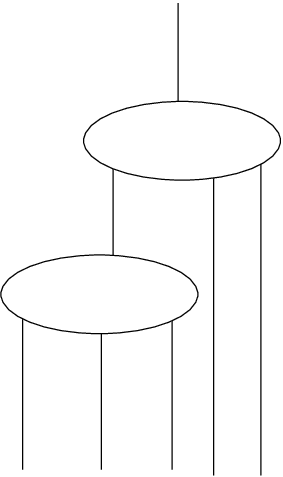}
\\
\\ a) $\mu^\CC=4$, 3 markings & b)  $\mu^\CC=1$, 23 markings & c)
$\mu^\CC=4$, 7 markings  & d)  $\mu^\CC=1$, 21 markings
\end{tabular}
\caption{Floor diagrams of genus 0 and Newton polygon  depicted in
  Figure \ref{ex poly}c}
\label{comp cplx2}
\end{figure}

\subsection{Enumeration of real curves}
First of all, we have to define the notion of real marked floor diagrams.
Like before, we define  $s=Card(\partial\Delta\cap
\ZZ^2)+g
-1$. Choose an integer $r\ge 0$ such that
$s-2r\ge 0$, and $\mathcal D$  a floor
diagram
of genus $0$ and  Newton polygon $\Delta$ marked by a map $m$.

The set $\{i,i+1\}$ is a called \textit{$r$-pair} if
 $i=s-2k+1$ with  $1\leq k \leq r$.
 Denote by
$\Im (m,r)$ the union of all the $r$-pairs $\{i,
i+1\}$ where $m(i)$ is not adjacent to  $m(i+1)$.
Let $\rho_{m,r}:\{1,\ldots, s\} \rightarrow \{1,\ldots, s\} $ be
the bijection
defined by  $\rho_{m,r} (i)=i$ if $i\notin
 \Im(m,r)$, and by $\rho_{m,r}(i)=j$ if
$\{i,j\}$ is a $r$-pair  contained in $ \Im(m,r)$.
Note that  $\rho_{m,r}$ is an involution.

We define $o_r$ to be the half of the number of vertices $v$ of $\D$ in
$m(\Im(m,r))$ with odd divergence $\text{div}(v)$,
 and we set $A= \text{Edge}(\mathcal D) \setminus  m(\{1,\ldots,
s-2r\})$.

\begin{defi}\label{defi real}
A marked floor diagram $(\mathcal D,m)$
 is called $r$-real if the two marked floor diagrams
$(\mathcal{D},m)$ and
$(\mathcal{D},m\circ \rho_{m,r})$ are
equivalent.

The $r$-real multiplicity of a  $r$-real marked floor
diagram, denoted by
$\mu^\RR_r(\mathcal D,m)$, is defined as
$$\mu^\RR_r(\mathcal D,m)= (-1)^{o_r}
\prod_{e\in A}w(e) $$
if  all
edges of $\D$ of even weight contains a point of $m(\Im(m,r))$, and as
$$\mu^\RR_r(\mathcal D,m)=0 $$
otherwise.

For convenience we set $\mu^\RR_r(\mathcal D,m)=0 $ also in the case when
$(\mathcal D,m)$ is not $r$-real.

\end{defi}

Note that $\mu^\RR_0(\mathcal D,m)= 1$ or $0$ and is equal to $\mu^\CC(\mathcal D)$  modulo 2,
hence doesn't depend on $m$. However,
  $\mu^\RR_r(\mathcal D,m)$ depends on $m$ as soon as $r\ge 1$.
Next theorem is the second main formula of this paper.

\begin{thm}\label{WFD}
Let  $\Delta$ be a $h$-transverse polygon such that Welschinger
invariants are defined for the corresponding
toric surface $Tor(\Delta)$ equipped with its tautological real structure. Then for
any integer $r$ such that $s-2r\ge 0$, one has
$$W(\Delta,r)=\sum \mu^\RR_r(\mathcal D,m)$$
where the sum is taken over all 
marked floor diagrams of
genus 0 and  Newton polygon $\Delta$.
\end{thm}
Theorem \ref{WFD} is a corollary of Proposition \ref{fd bij} proved in
section \ref{Floor trop}.

\begin{exa}
All marked floor diagrams of  genus 0 and  Newton polygon $\Delta_3$
  are depicted in
Table \ref{cubic} together with their real multiplicities. The first
floor diagram
has an edge of weight 2, but we didn't mention it in the picture to
avoid confusion.
According to Theorem \ref{WFD} we find $W(\Delta_3,r)=8-2r$.
\end{exa}

\begin{table}[h]
\centering
\begin{tabular}{c|c|c|c|c|c|c|c|c|c}
&
\includegraphics[height=2.5cm, angle=0]{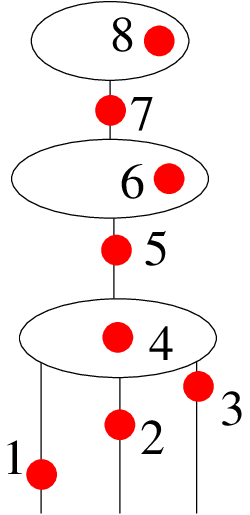}&
\includegraphics[height=2.5cm, angle=0]{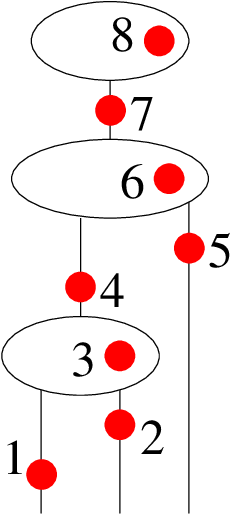}&
\includegraphics[height=2.5cm, angle=0]{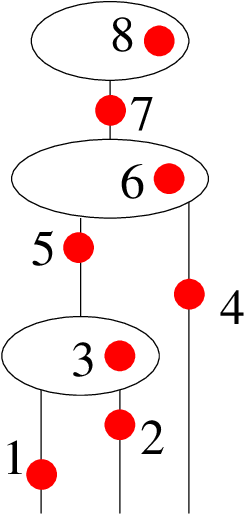}&
\includegraphics[height=2.5cm, angle=0]{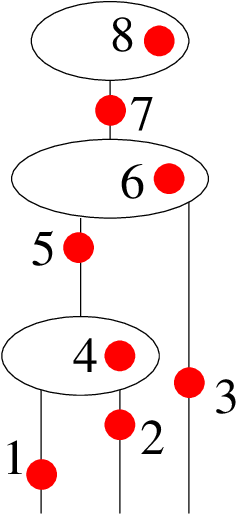}&
\includegraphics[height=2.5cm, angle=0]{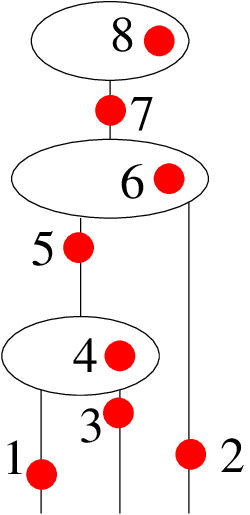}&
\includegraphics[height=2.5cm, angle=0]{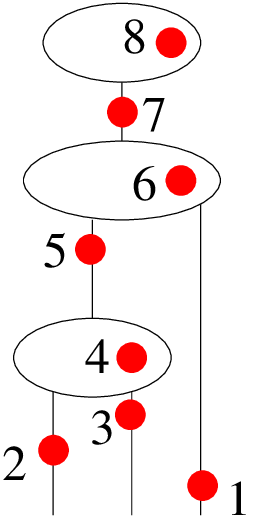}&
\includegraphics[height=2.5cm, angle=0]{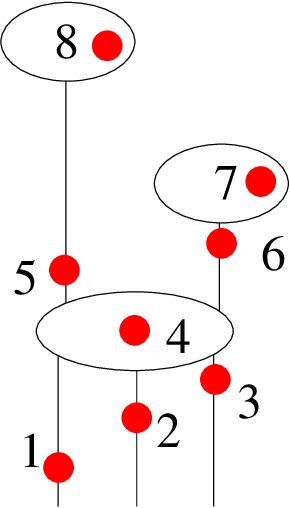}&
\includegraphics[height=2.5cm, angle=0]{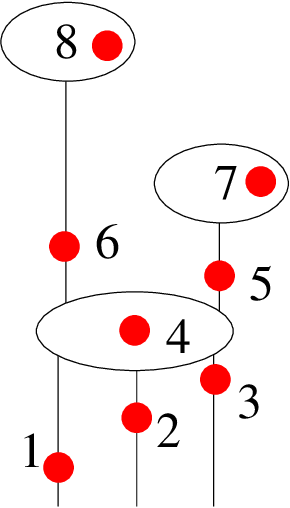}&
\includegraphics[height=2.5cm, angle=0]{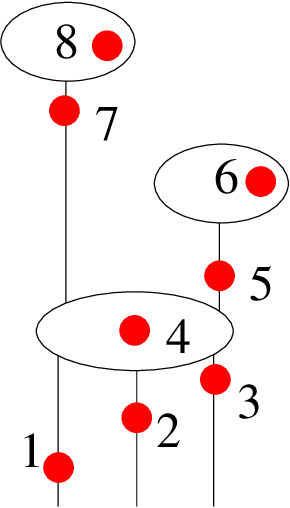}
\\ \hline $ \mu^\CC$ & 4 & 1 &1 &1 &1 &1 &1 &1 &1
\\ \hline $ \mu^\RR_0$ & 0 & 1 &1 &1 &1 &1 &1 &1 &1
\\ \hline $ \mu^\RR_1$ & 0 & 1 &1 &1 &1 &1 &0 &0 &1
\\ \hline $ \mu^\RR_2$ & 0 & 1 &1 &1 &1 &1 &-1 &-1 &1
\\ \hline $ \mu^\RR_3$ & 0 & 1 &0 &0 &1 &1 &-1 &-1 &1
\\ \hline $ \mu^\RR_4$ & 0 & 1 &0 &0 &0 &0 &-1 &-1 &1

\end{tabular}
\\
\begin{tabular}{c}
\end{tabular}
\caption{Computation of $W(\Delta_3,r)$}
\label{cubic}
\end{table}

\section{Enumerative tropical geometry}\label{enum trop geo}

\subsection{Tropical curves}\label{defi trop curve}
\begin{defi}
An irreducible  tropical curve $C$ is a connected compact metric graph
whose leaves are exactly the edges of infinite length.
This means that $C\setminus End(C)$ is a complete metric space with
inner metric. In other words the 1-valent vertices are at the infinite
distance from all the other points of $C$.
The genus of $C$ is defined as its first Betti number $b_1(C)$.
\end{defi}

\begin{exa}
Examples of tropical curves are depicted in Figure \ref{ex
  trop}. 1-valent vertices are represented with bullets.
\end{exa}

\begin{figure}[h]
\centering
\begin{tabular}{ccccc}
\includegraphics[height=3cm, angle=0]{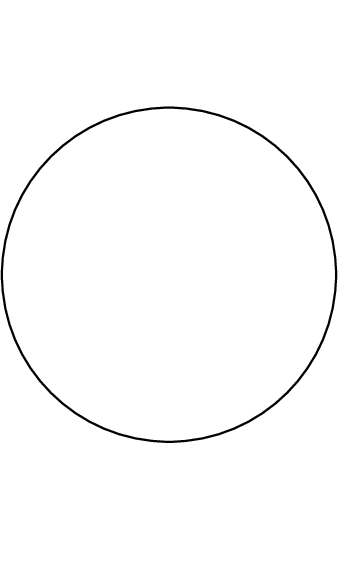}& \hspace{3ex} &
\includegraphics[height=3cm, angle=0]{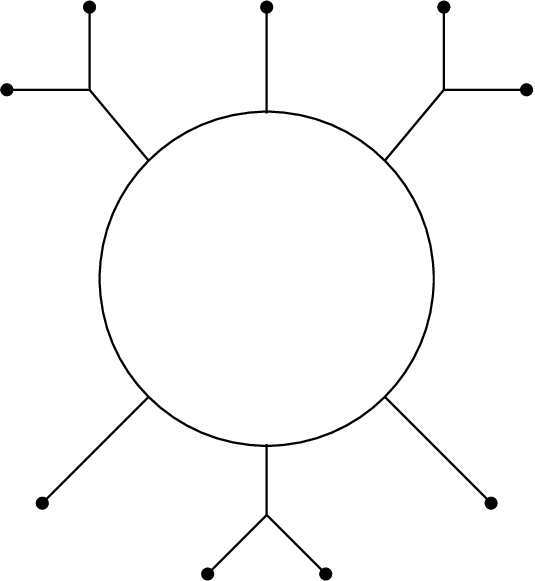}&\hspace{3ex} &
\includegraphics[height=3cm, angle=0]{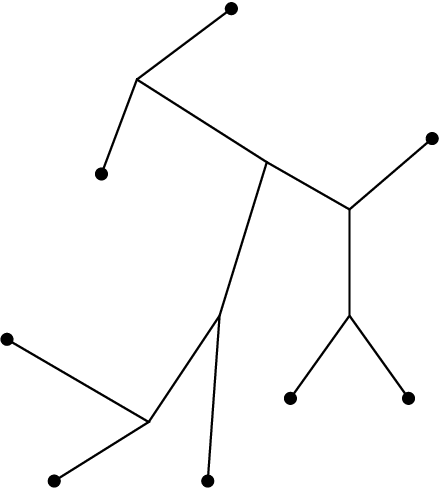}
\\
\\a) $g=1$  &&  b) $g=1$ && c)  $g=0$
\end{tabular}
\caption{Examples of tropical curves}
\label{ex trop}
\end{figure}

Given $e$  an edge of a tropical curve $C$, we choose a point $p$ in
the interior of $e$ and
a unit vector $u_e$
of the tangent line to $C$ at $p$. Of course,  the vector $u_e$
depends on the choice of $p$ and is not well defined, but this will not
matter in the following. We will sometimes need $u_e$ to have a
prescribed direction, and we will then precise this direction.
The
standard inclusion of $\ZZ^2$ in $\RR^2$ induces a standard
inclusion of $\ZZ^2$ in the tangent space of $\RR^2$  at any point of $\RR^2$.

\begin{defi}
A map $f : C\setminus\text{End}(C) \to \RR^2$ is called a tropical
morphism if the following conditions are satisfied
\begin{itemize}
\item for any edge $e$ of $C$, the restriction $f_{|e}$ is a
  smooth map with $df(u_e)=w_{f,e}u_{f,e}$ where
  $w_{f,e}$ is a non-negative integer and $u_{f,e}\in\ZZ^2$ is a
  primitive vector,

\item for any vertex $v$ of $C$ whose adjacent  edges are
  $e_1,\ldots,e_k$,
  one has the balancing condition
$$\sum_{i=1}^k  w_{f,e_i}u_{f,e_i}=0 $$
where $u_{e_i}$ is chosen so that it points away from $v$.
\end{itemize}

\end{defi}

Let $f : C\setminus\text{End}(C) \to \RR^2$ be  a tropical morphism, and
define $L(C,f)$ as  the unordered list composed by elements $u_{f,e}$
repeated $w_{f,e}$ times where
 $e$ goes through leaves of $C$ and $u_{e}$ is chosen so that it
points to the 1-valent vertex.  Then, there exists a
unique, up to translation by a vector in $\ZZ^2$, lattice polygon
$\Delta(C,f)$ such that the unordered list composed by the primitive
vector normal to $e$ and  outward to $\Delta(C,f)$ repeated $l(e)$
times where $e$ goes through
edges of $\Delta(C,f)$ equals the list $L(C,f)$.

\begin{defi}
The polygon $\Delta(C,f)$ is called the Newton polygon of the
pair $(C,f)$.
\end{defi}

Not any tropical curve admits a non-constant tropical morphism to $\RR^2$. The
tropical curve
depicted in Figure \ref{ex trop}a does not admit any tropical morphism
since a circle cannot be mapped to a segment in $\RR^2$ by a
dilatation. However, up to \textit{modification}, every tropical curve
can be tropically immersed to $\RR^2$ (see \cite{Mik08}).

The pair $(C,f)$ where $f : C\setminus\text{End}(C) \to
\RR^2$ is a tropical morphism with Newton polygon $\Delta$ is  called
a \textit{parameterized
  tropical curve with Newton polygon $\Delta$}. The integer $w_{f,e}$
 is
called the \textit{weight} of the edge $e$. The genus of $(C,f)$ is
naturally defined as the genus of $C$.

\begin{exa}
If $C$ is the tropical curve depicted in
Figure \ref{ex trop}b (resp.  c) then
an example of the image $f(C)$ for some parameterization
 with Newton polygon $\Delta_3$ is depicted in Figure \ref{ex trop
  morphism}a (resp. b). The second tropical morphism has an edge of
 weight 2.
\end{exa}

\begin{figure}[h]
\centering
\begin{tabular}{ccc}
\includegraphics[height=4cm, angle=0]{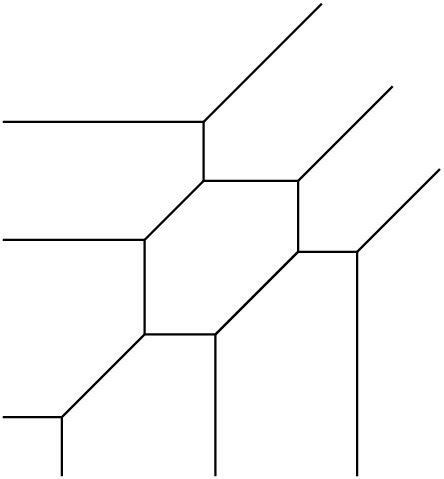}& \hspace{3ex} &
\includegraphics[height=4cm, angle=0]{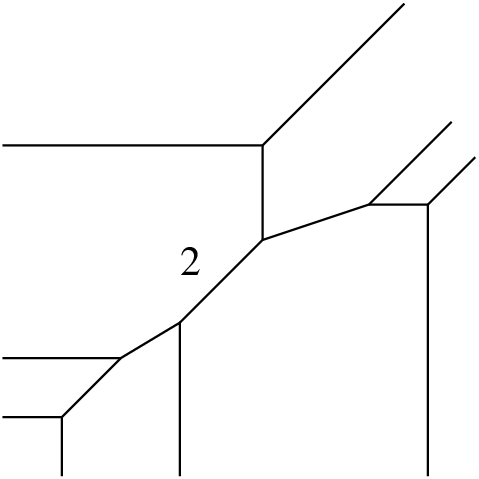}
\\
\\a)  &&  b)
\end{tabular}
\caption{Images of tropical morphisms with Newton polygon $\Delta_3$}
\label{ex trop morphism}
\end{figure}

\begin{defi}
A tropical curve  with $n$ marked points is a $(n+1)$-tuple
$(C,x_1,\ldots,x_n)$ where $C$ is a tropical curve and the $x_i$'s are
$n$  points on $C$.

A parameterized tropical curve with $n$ marked points is a $(n+2)$-tuple
$(C,x_1,\ldots,x_n,f)$ where $(C,x_1,\ldots,x_n)$ is a tropical curve
with $n$ marked points, and $(C,f)$ is a parameterized tropical curve.
\end{defi}
Note that in this paper we do not require the marked points on a marked tropical curve
to be distinct.
In the following, we consider tropical curves (with $n$ marked points)
up to homeomorphism of metric graphs (which send the $i^{th}$ point to the
$i^{th}$ point).
The notions of vertices, edges, Newton polygon, $\ldots$ also make sense
for a parameterized marked tropical curve as the corresponding notions
for the underlying (parameterized) tropical curve.

\subsection{Complex multiplicity of a tropical curve}
Let us now turn to tropical enumerative geometry, and let's relate it
first to complex enumerative geometry. More details about this section can
be found in \cite{Mik1} or \cite{GM1}.

Fix a lattice
polygon $\Delta$, a non-negative integer number $g$, and define
$s=Card(\partial \Delta\cap \ZZ^2) -1 +g$. Choose
a collection $\omega=\{p_1,\ldots, p_s\}$
of $s$
points in $\RR^2$, and denote by
$\C(\omega)$ the set of parameterized tropical curves with $s$-marked points
$(C,x_1,\ldots, x_s,f)$ satisfying the following conditions
\begin{itemize}
\item the tropical curve $C$ is irreducible and of genus $g$,
\item  $\Delta(C,f)=\Delta$,
\item for any $1\le i\le s$, $f(x_i)=p_i$.
\end{itemize}

\begin{figure}[h]
\begin{center}
\begin{tabular}{ccccccc}
\includegraphics[height=6cm,
  angle=0]{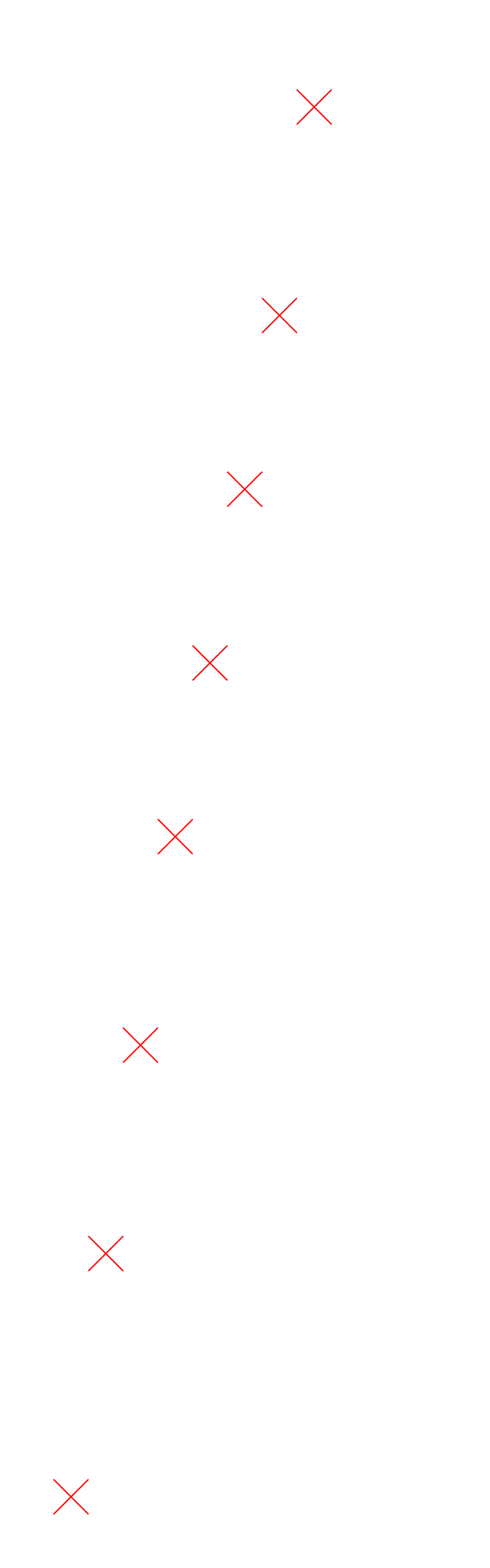}& \hspace{3ex} &
\includegraphics[height=6cm, angle=0]{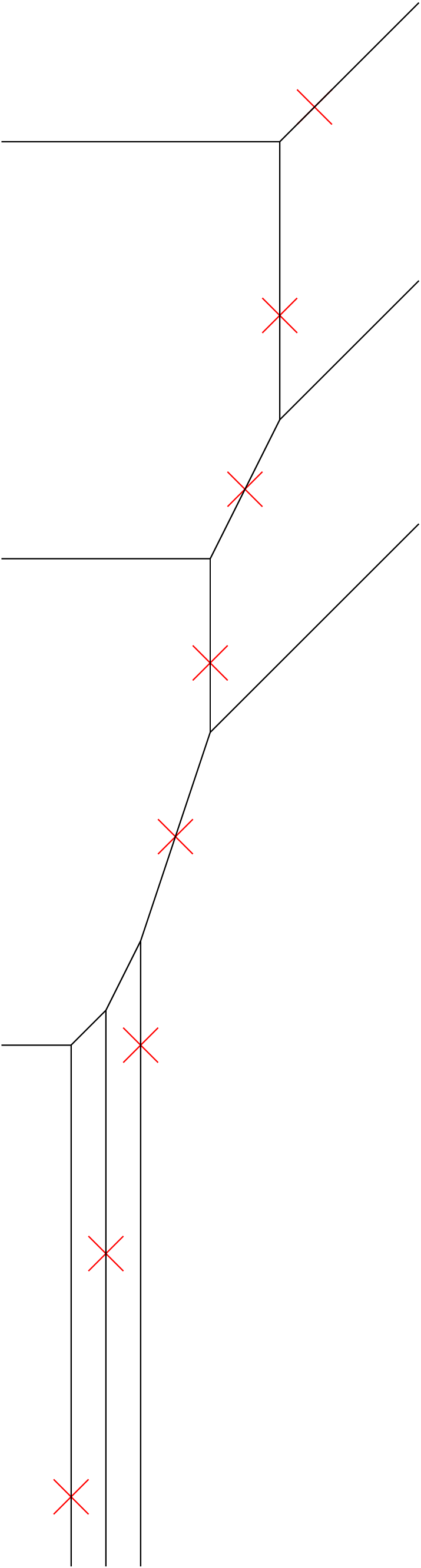}&
\includegraphics[height=6cm, angle=0]{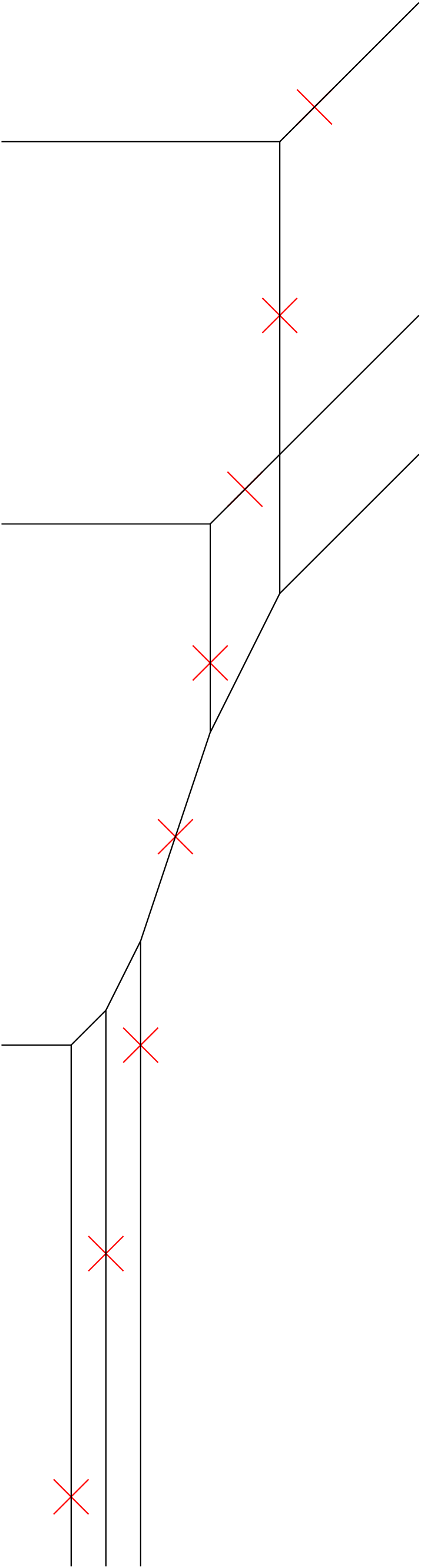}&
\includegraphics[height=6cm, angle=0]{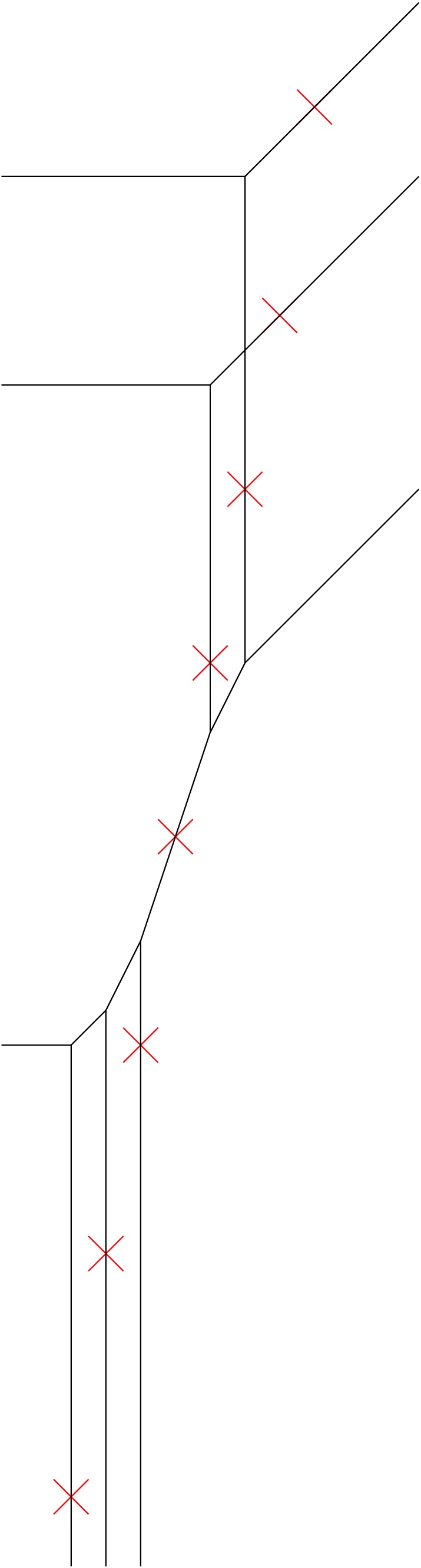}&
\includegraphics[height=6cm, angle=0]{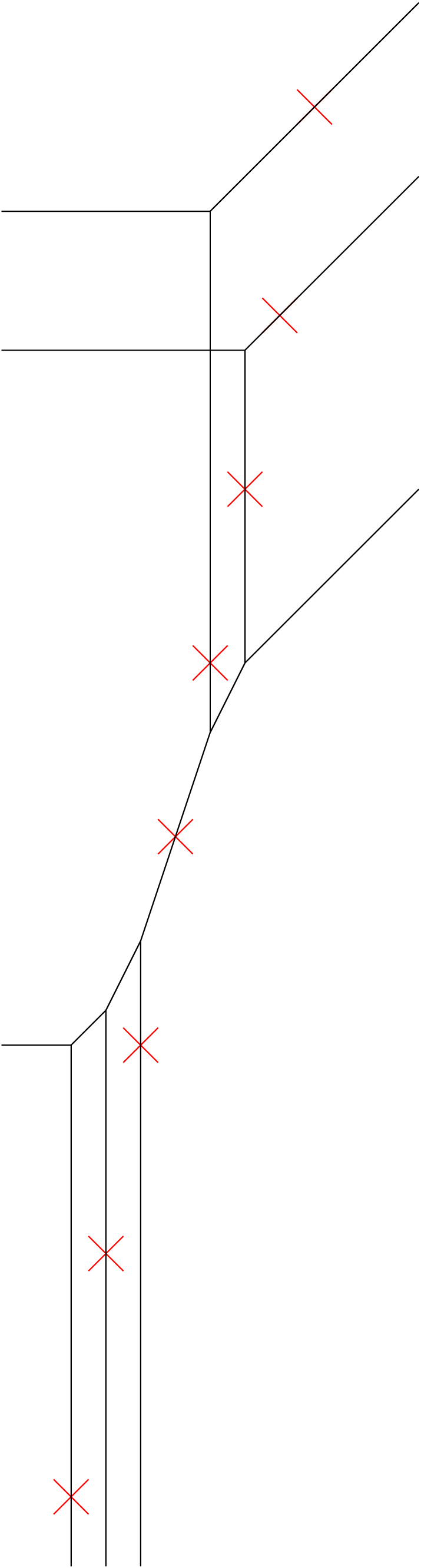}&
\includegraphics[height=6cm, angle=0]{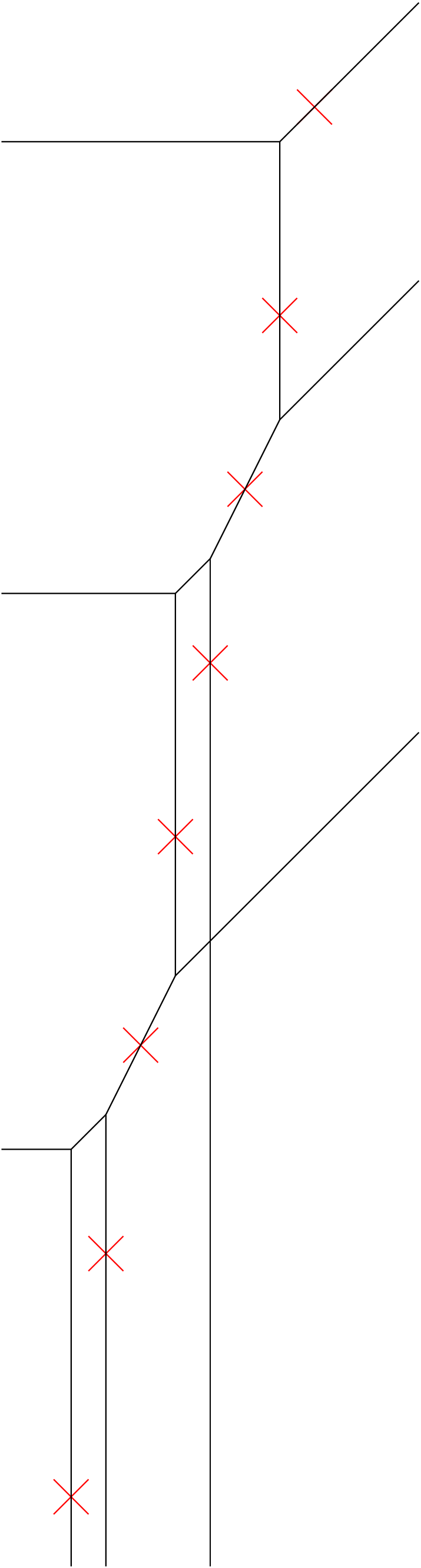}
\\
\\ a)&& b) $\mu^\CC=4$ &c)  $\mu^\CC=1$ &d) $\mu^\CC=1$ &
e) $\mu^\CC=1$ &f) $\mu^\CC=1$
\\
\\

\end{tabular}

\begin{tabular}{cccc}
\includegraphics[height=6cm, angle=0]{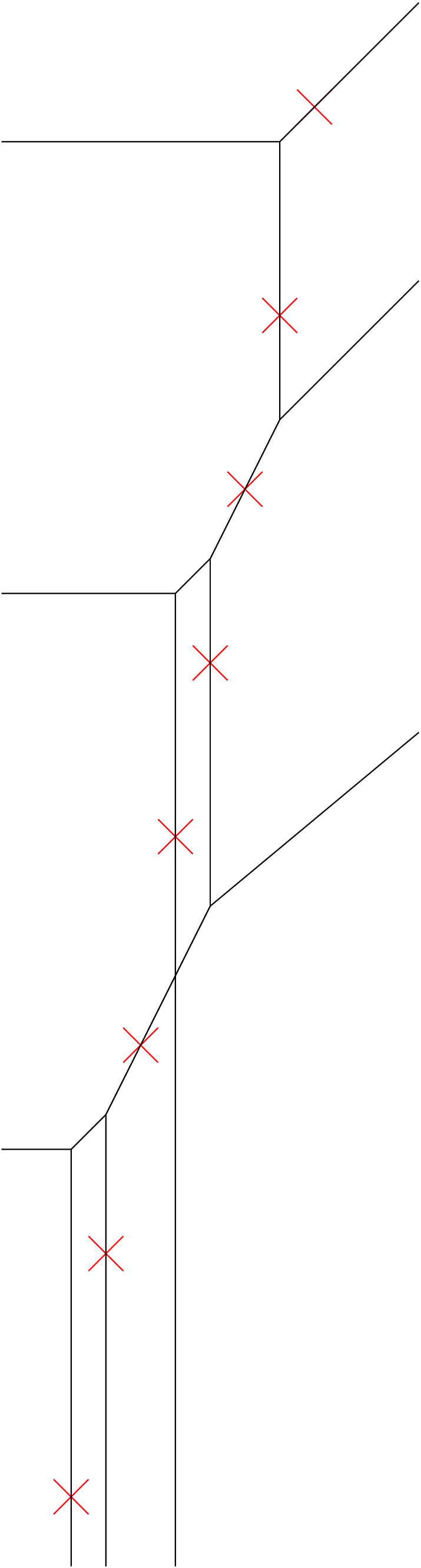}&
\includegraphics[height=6cm, angle=0]{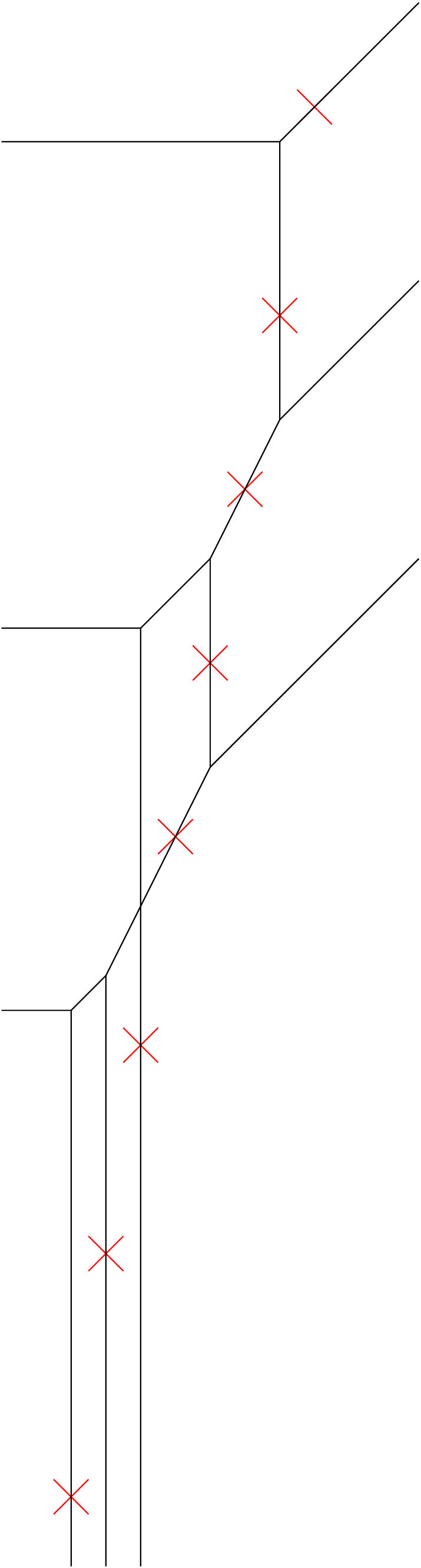}&
\includegraphics[height=6cm, angle=0]{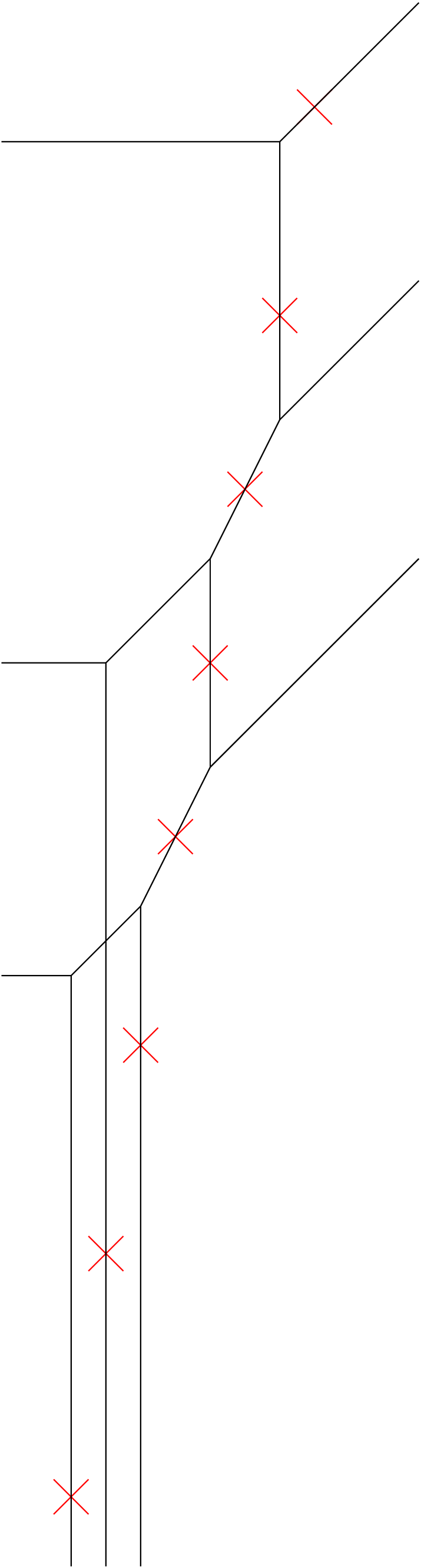}&
\includegraphics[height=6cm, angle=0]{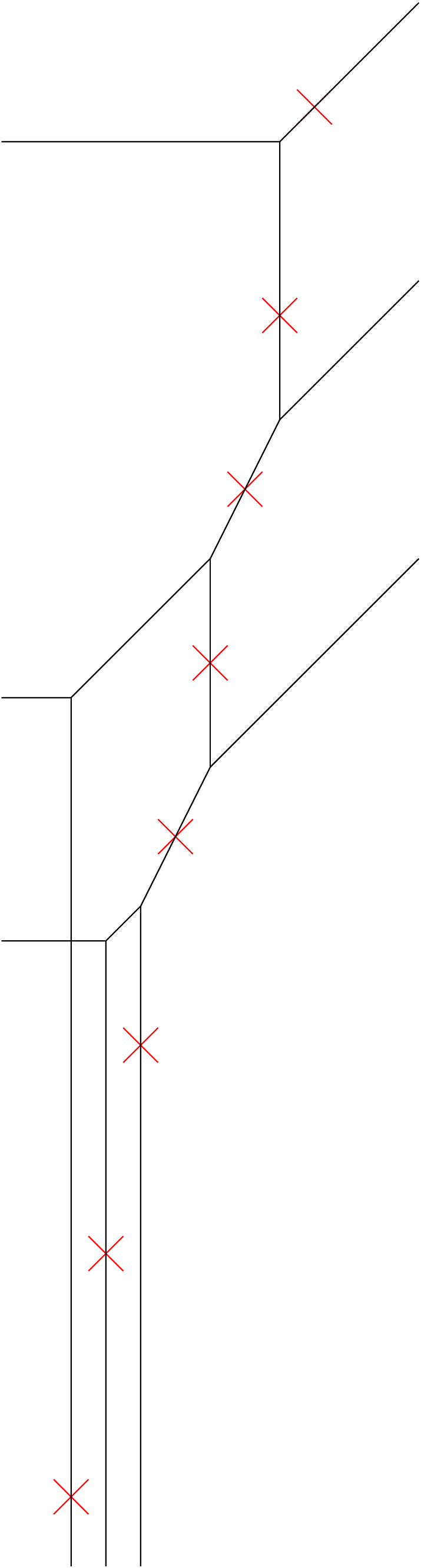}
\\
\\ g)  $\mu^\CC=1$ &h) $\mu^\CC=1$ &
i) $\mu^\CC=1$ &j) $\mu^\CC=1$
\\

\end{tabular}

\end{center}
\caption{$N(\Delta_3,0)=12$}
\label{Tcub0}
\end{figure}

\begin{prop}[Mikhalkin, \cite{Mik1}]\label{finite 1}
For a generic configuration of points $\omega$, the set $\C(\omega)$ is
finite. Moreover, for any parameterized tropical curve $(C,x_1,\ldots,
x_s,f)$ in $\C(\omega)$, the curve $C$
has only 1 or 3-valent  vertices, the set $\{x_1,\ldots,
x_s\}$ is disjoint from $\text{Vert}(C)$, any
leaf of
$C$  is of weight 1, and $f$ is a topological immersion.
In particular, any neighborhood of any
3-valent vertex of $C$ is never mapped to a segment by $f$.
\end{prop}

Given a generic configuration $\omega$,  we associate a complex multiplicity $\mu^\CC(\widetilde
C)$ to any element $\widetilde C=(C,x_1,\ldots, x_s,f)$ in
$\C(\omega)$. Let
$v$ be a vertex of $C\setminus \text{End(C)}$ and $e_1$ and
$e_2$ two of its adjacent edges. As $v$ is trivalent, the balancing
condition implies that the number
$\mu^\CC(v,f)=w_{f,e_1}w_{f,e_2}|\det(u_{f,e_1},u_{f,e_2})|$ does not depend on
the choice of $e_1$ and $e_2$.

\begin{defi}
The complex multiplicity of an element $\widetilde C$ of
$\C(\omega)$,
denoted by $\mu^\CC(\widetilde C)$, is defined as
$$\mu^\CC(\widetilde C) =\prod_{v\in \text{Vert}(\widetilde C)} \mu^\CC(v,f)$$
\end{defi}

\begin{thm}[Mikhalkin, \cite{Mik1}]\label{corr 1}
For any lattice polygon $\Delta$, any genus $g$, and any
generic configuration $\omega$ of $Card(\partial \Delta\cap\ZZ^2) -1 +g$ points
in $\RR^2$, one has
$$N(\Delta, g)=\sum_{\widetilde C\in \C(\omega)} \mu^\CC(\widetilde C) $$
\end{thm}

\begin{exa}
Images $f(C)$ of all irreducible tropical curves of genus 0 and
Newton polygon $\Delta_3$ in $\C(\omega)$ for the  configuration
$\omega$ of 8 points  depicted in Figure \ref{Tcub0}a are depicted in
Figure \ref{Tcub0}b, $\ldots$, j. We verify
that $N(\Delta_3,0)=12$ (compare with Table \ref{cubic}).
\end{exa}

\subsection{Real multiplicities of a tropical curve}

We explain now how to adapt Theorem \ref{corr 1} to real enumerative geometry.
Naturally, we need to consider tropical curves endowed with a real structure.

\begin{defi}
A real parameterized tropical curve with $n$ marked points is a  $(n+3)$-uplet
$(C, x_1, \ldots, x_n, f, \phi)$ where $(C,x_1,\ldots,x_n,f)$ is a
parameterized marked  tropical curve and  $\phi : C\to C$ is an isometric
involution
such that
\begin{itemize}
\item there exists a permutation $\sigma$ such that
  for any $1\le i\le
  n$,  $\phi(x_i)=x_{\sigma(i)}$,
\item $f=f\circ\phi$.
\end{itemize}
\end{defi}

The \textit{real and imaginary parts} of a real parameterized tropical curve
with $n$ marked
points $\widetilde C=(C, x_1, \ldots, x_n, f, \phi)$ are naturally defined as
$$\Re(\widetilde C)=\text{Fix}(\phi) \ \ \ \text{and}
\ \ \ \Im(\widetilde C)=C\setminus\Re(\widetilde C)$$

\begin{exa}
Two examples of real parameterized tropical curves with 4 marked points are
depicted in Figure
\ref{real curves}, the abstract curve is depicted on the left and its
image by $f$ in $\RR^2$ is depicted on the right. Very close edges in
the image represent edges which are mapped to the same edge by $f$.
The parameterized tropical curve in Figure \ref{real curves}a
has 2 equal marked points, and $\phi$ is the symmetry with
respect to the non-leaf edge. In Figure
\ref{real curves}b,
$\phi$  exchanges the edges containing $x_1$ and $x_2$.
\end{exa}
\begin{figure}[h]
\centering
\begin{tabular}{ccc}
\includegraphics[height=5cm,
  angle=0]{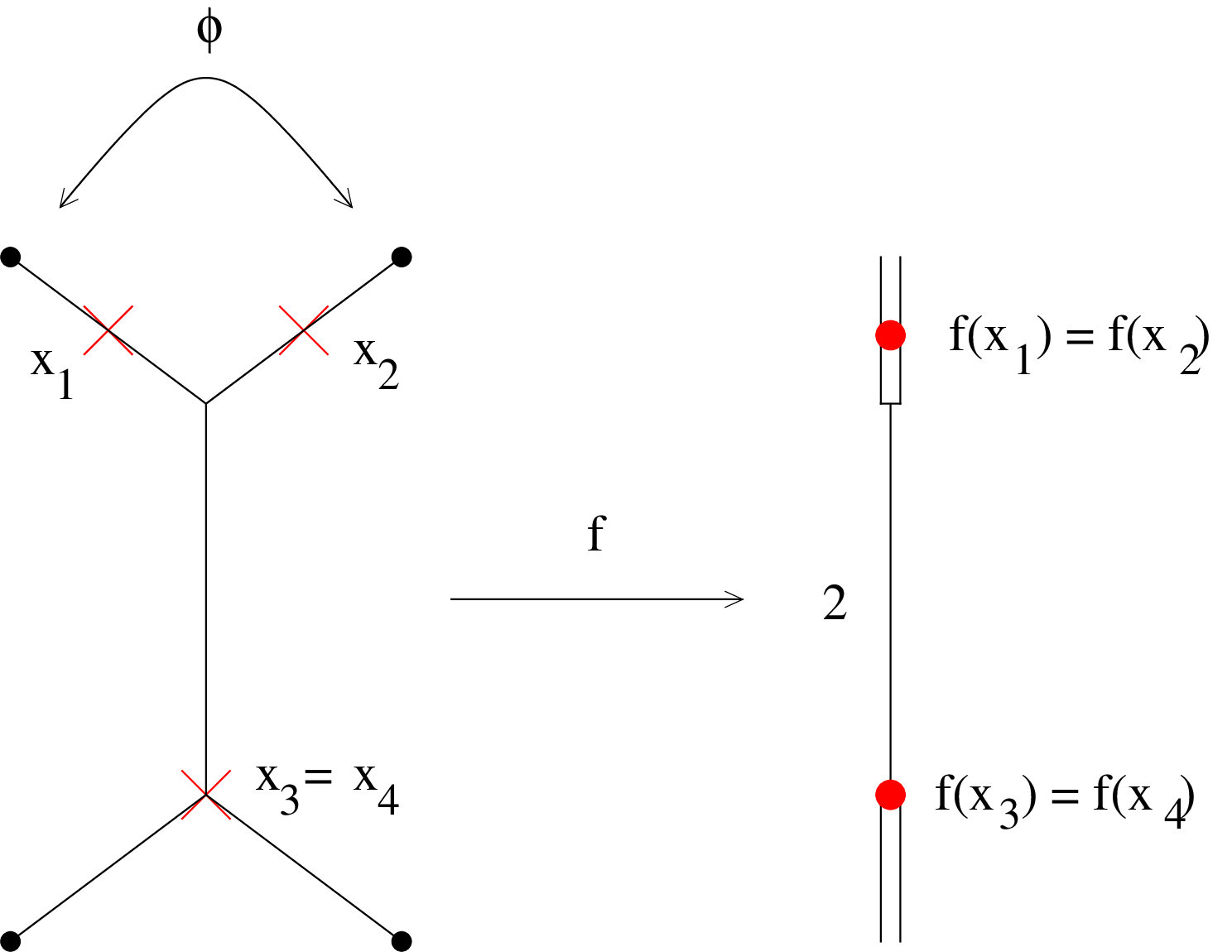}& \hspace{3ex} &
\includegraphics[height=5cm, angle=0]{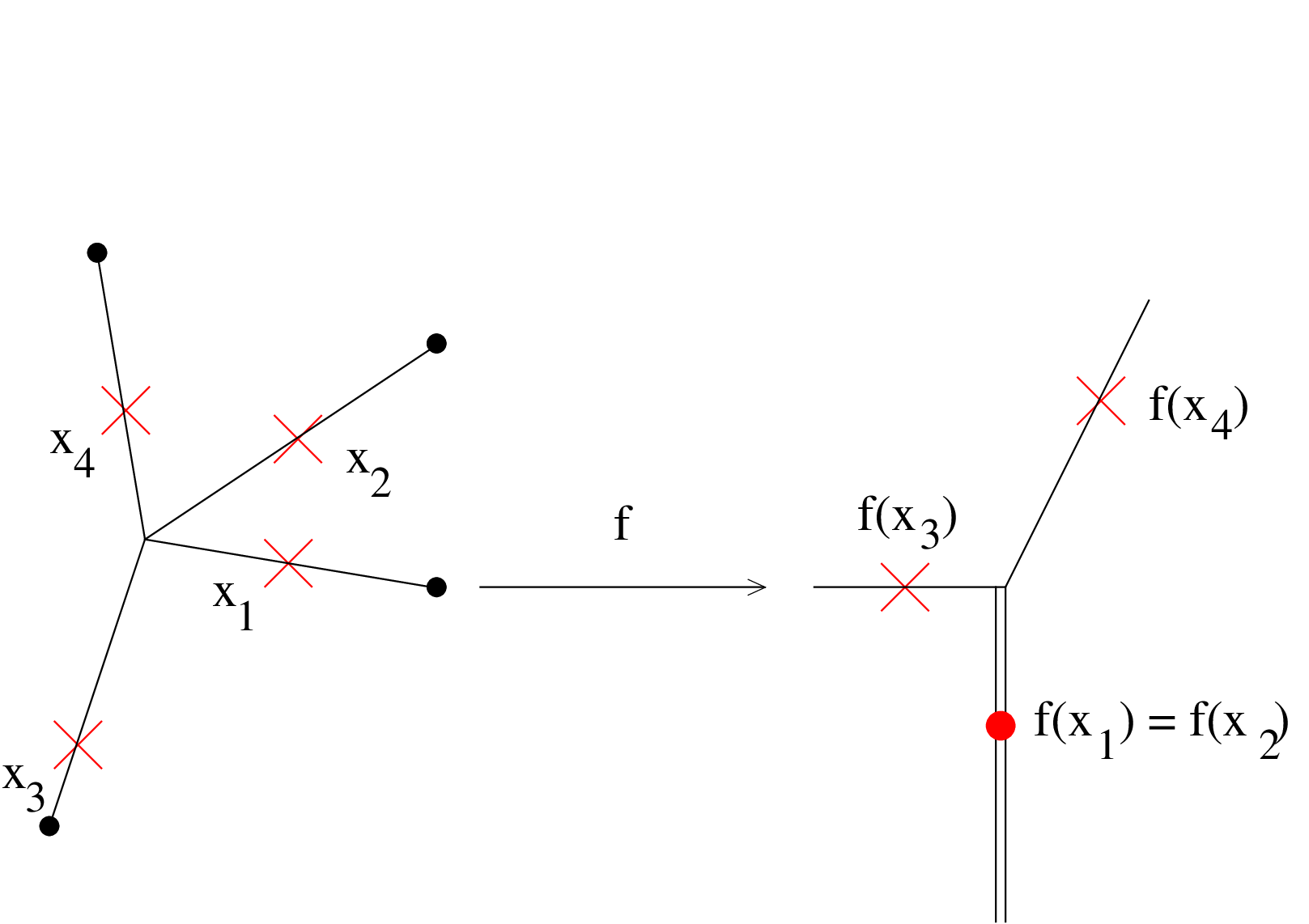}
\\
\\a) && b)
\end{tabular}
\caption{Real tropical curves}
\label{real curves}
\end{figure}

As usual, we fix a lattice
polygon $\Delta$ and define
$s=Card(\partial \Delta\cap \ZZ^2) -1$.
Let $r$ be an non-negative integer such that $s-2r\ge 0$,  and
choose
a collection $\omega_r=\{p_1,\ldots, p_{s-r}\}$
of $s-r$
points in $\RR^2$. We should think of  $\omega_r$ as the image under
the map $(z,w)\mapsto (\log|z|,\log|w|)$ of a configuration
$\{q_1,\ldots,q_{s-2r},q_{s-2r+1},\overline{q_{s-2r+1}},\ldots
,q_{s-r},\overline{q_{s-r}}\}$ of $s$
points in $(\CC^*)^2$, where $\overline z$ is the complex conjugated
of $z$. Hence, points $p_i$ with $s-2r+1\le i\le s-r$ represent pairs of
complex conjugated points.
 Denote by  $\RR\C(\omega_r)$ the set of irreducible real parameterized tropical
curves with $s$ marked points $\widetilde C=(C, x_1, \ldots, x_s, f, \phi)$
of genus 0 and Newton polygon $\Delta$ satisfying the following conditions

\begin{itemize}
\item for any $1\le i\le s-2r$, $f(x_i)=p_i$,
\item for any $1\le i \le r$,
  $f(x_{s-2r+2i-1})=f(x_{s-2r+2i})=p_{s-2r+i}$,
\item if $1\le i \le r$ and if
   $x_{s-2r+2i-1}=  x_{s-2r+2i} $, then  $x_{s-2r+2i}$ is a
   vertex of $C$,
\item any edge in $\Re(\widetilde C)$ has an odd weight.
\end{itemize}

\begin{prop}\label{finite 2}
For a generic configuration of points $\omega_r$, the set
$\RR\C(\omega_r)$ is
finite. Moreover, for any real parameterized curve $\widetilde C=(C, x_1, \ldots,
x_s, f, \phi)$ in  $\RR\C(\omega_r)$, the curve $C$ has only 1, 3 or 4
valent
vertices,  any neighborhood of any
3 or 4-valent vertex of $C$ is never mapped to a segment by $f$, any
4-valent vertex of $C$ is adjacent to 2 edges in
$\Im(\widetilde C)$ and 2 edges in $\Re(\widetilde C)$,
 and any leaf of
$C$ is of weight 1.
\end{prop}
\begin{proof}
Let $\widetilde C$ be an element of  $\RR\C(\omega_r)$.
Passing through $s-r$ points in $\RR^2$ in general position imposes
$2(s-r)$ independent  conditions on a tropical curve.
Since all tropical maps are piecewise-linear, to prove
the proposition it suffices to show that the dimension of the space
of all real parameterized
tropical curves with the same combinatorial type as $\widetilde C$
has dimension  $2(s-r)$, and that any curve
with this combinatorial type satisfies the proposition.

Recall that the space
of all parameterized irreducible
tropical curves $(C,f)$ of genus 0 with $x$ leaves and of a given
combinatorial type is a
polyhedral complex of dimension
$$x-1 -
\sum_{v\in\text{Vert}(C)\setminus \text{End}(C)} (\text{val}(v)-3) -n_c$$
where $\text{val}(v)$ is the valence of a vertex $v$, and $n_c$ is the
number of  edges of $C$  contracted by $f$ (see \cite{Mik1}).
Let $\widetilde C=(C, x_1, \ldots,
x_s, f, \phi)$ be an element of  $\RR\C(\omega_r)$.
We may prepare two auxiliary tropical curves $f^r:C^r\to\RR^2$ and
$f^i:C^i\to\RR^2$
from $f:C\to\RR^2$. We say that $v\in C$ is a {\em junction} vertex if
any small neighborhood
of $v$ intersects both $\Re (\widetilde C)$ and $\Im (\widetilde C)$. We denote by $J$ the
number of junction vertices of $C$. Since any edge of $\Re (\widetilde
C)$ has odd weight, a junction vertex is at least 4-valent.

We define $C^r$ to be the result of adding to $\Re (\widetilde C)$ an
infinite ray at
each junction vertex of $C$. We define $f^r$ so that it coincides with $f$ on
$\Re (\widetilde C)$. The values of $f^r$ at the new rays are determined by the
balancing condition.

Connected components of $\Im(\widetilde C)$
are naturally coupled in pairs exchanged by the map $\phi$. To define
$C^i$, we take
$\Im (\widetilde C)/\phi$ and replace all  edges adjacent to a
junction vertex with an infinite ray.
We let $f^i:C^i\to\RR^2$ to be the tropical map that
agrees with $f$ on $\Im (\widetilde C)/\phi$. We denote by $n^i$ the number of
connected components of $C^i$. Note that $n^i\ge J$, and that equality
holds if and only if each junction vertex is 4-valent.

We denote by $x^r$ (resp. $x^i$) the number of
leaves of $C$ which are also leaves of  $C^r$ (resp. $C^i$).
Since the curve $C$ has genus 0, the curve $C^r$ is connected and each component
of $C^i$ is adjacent to exactly one junction vertex. Hence, the space
of  parameterized
tropical curves  with the same combinatorial type as $(C^r,f^r)$ has
dimension
$$x^r + J -1 -
\sum_{v\in\text{Vert}(C^r)\setminus \text{End}(C^r)} (\text{val}(v)-3) -n_{c^r}$$
and the space
of  parameterized
tropical curves  with the same combinatorial type as $(C^i,f^i)$ has
dimension
$$x^i  -
\sum_{v\in\text{Vert}(C^i)\setminus \text{End}(C^i)} (\text{val}(v)-3) -n_{c^i}$$
To get $f$ from $f^r$ and $f^i$ these maps must agree at each junction. Thus
each connected component of $C^i$ imposes one condition, and
the space
of  real parameterized
tropical curves  with the same combinatorial type as $(C,f)$ has
dimension
$$x^r +x^i + J - n^i -1 -
\sum_{v\in\text{Vert}(C^r)\setminus \text{End}(C^r)} (\text{val}(v)-3) -
\sum_{v\in\text{Vert}(C^i)\setminus \text{End}(C^i)}(\text{val}(v)-3) -n_{c^r} -n_{c^i}$$
If we consider, in addition, a configuration of $s$ points on $C$ then
our dimension
increases by $s$. Recall though that our points are constrained by the
condition that the last $2r$ points are split into pairs invariant with
respect to the involution $\phi$. Furthermore, recall that if such a
pair consists of the same point taken twice then it must be a vertex of $C$.

Denote with $p$ the number of pairs of distinct points in $C$
invariant with respect to $\phi$ and with $q$ the number of pairs
made from the vertices of $C$. Clearly we have $p+q=r$, and
the dimension of allowed configurations
is $s-p-2q$.
Since $f^i(C^i)$ passes through $p$ generic points in $\RR^2$, we have
$x_i\ge p$, and since $x^r+2x^i\le s+1$ we have $x^r+x^i \le s+1-p$.
Hence, the space
of  real parameterized
tropical curves  with $s$ marked points with the same combinatorial
type as $\widetilde C$ has
dimension at most
$$2(s-r) + J - n^i  -
\sum_{v\in\text{Vert}(C^r)\setminus \text{End}(C^r)} (\text{val}(v)-3) -
\sum_{v\in\text{Vert}(C^i)\setminus \text{End}(C^i)}(\text{val}(v)-3)
-n_{c^r} -n_{c^i}$$
Since $\widetilde C$ is in  $\RR\C(\omega_r)$, its space of deformation must have dimension
at least $2(s-r)$. Hence the curve $C$ has exactly $s+1$ leaves, $
n^i=J$, and any
vertex of $C$ which
is not an end or a junction vertex is trivalent.
\end{proof}

For a generic configuration $\omega_r$  and $(C, x_1, \ldots,
x_s, f,\phi)$ in $\RR\C(\omega_r)$, Proposition \ref{finite 2} implies that the real structure
$\phi$ on $C$ is uniquely determined by the marked parameterized
tropical curve $(C, x_1, \ldots,
x_s, f)$. Hence we will often omit to precise the map $\phi$ for
elements of
$\RR\C(\omega_r)$.
 Moreover,
$\Im(\widetilde C)/\phi$ is a (possibly disconnected) non-compact graph, and
a vertex $v$ (resp. edge) inside $\Im(\widetilde C)/\phi$ has a natural
complex
multiplicity
$\mu^\CC(v,f)$ (resp. weight) induced by the corresponding
multiplicity of vertices (resp. edges) of
$C$.
 If $v$ is a 4-valent vertex of $C$,
then by Proposition \ref{finite 2}, there exist an edge
$e_1\in\Re(\widetilde C)$ and an edge $e_2\in\Im(\widetilde C)$
adjacent to $v$.
Define $\mu^\RR_r(v,f)=w_{f,e_1}w_{f,e_2}|det(v_{f,e_1},v_{f,e_2})|$.
Define the
integer $o_r^\RR$ to be the number of vertices $v$ in
$\Re(\widetilde C)$
satisfying one of the following conditions
\begin{itemize}
\item $v$ is 3-valent and $\mu^\CC(v,f)=3\ mod \ 4$,
\item  $v$ is 4-valent adjacent to an edge $e\in\Im(\widetilde C)$, and
  $\mu^\RR(v,f)=w_{f,e}+1 \ mod \ 2$.
\end{itemize}

Finally,
define the integer $o_r^\CC$ to be the  number of vertices $v$ of
$\Im(\widetilde C)/\phi$ with odd $\mu^\CC(v,f)$.

\begin{defi}
The $r$-real  multiplicity of an element $\widetilde C $ of $\RR
\C(\omega_r)$,
denoted by $\mu^\RR_r(\widetilde
C)$, is defined as
$$\mu^\RR_r(\widetilde C) =(-1)^{o_r^\RR+o_r^\CC}\prod_{v\in
  \text{Vert}(\Im(\widetilde C)/\phi)} \mu^\CC(v,f)
\prod_{v\in
  \text{Vert}(\widetilde C), \ f(v)\in\omega_r} \mu^\CC(v,f)
\prod_{v\in
  \text{Vert}(\widetilde C),\ v \text{ is 4-valent}} \mu^\RR(v,f)$$
\end{defi}

The tropical curves and their multiplicity we are considering here
differ slightly from the one in \cite{Sh8}. This difference comes
from the fact that we are dealing with parameterization of tropical
curves, and that Shustin deals with the cycles resulting as the images of the curves
rather than parameterized curves.

\begin{rem}
If $r=0$, then for any real parameterized curve $(C, x_1, \ldots,
x_s, f, \phi)$ in  $\RR\C(\omega_r)$, we have $\phi=Id$, and
the map $(C, x_1, \ldots,
x_s, f, \phi)\mapsto (C, x_1, \ldots,
x_s, f)$ is a bijection from the set $\RR\C(\omega_r)$ to the set of
elements of $\C(\omega_r)$ with odd complex multiplicity.
\end{rem}

\begin{thm}[Mikhalkin, \cite{Mik1}, Shustin, \cite{Sh8}]\label{corr 2}
Let $\Delta$ be a lattice polygon such that Welschinger
invariants are  defined for the corresponding
toric surface $Tor(\Delta)$ equipped with its tautological real structure. Then for
any integer $r$ such that $s-2r\ge 0$,   and any
generic configuration $\omega_r$ of $s -r$ points
in $\RR^2$, one has
$$W(\Delta, r)=\sum_{\widetilde C\in \RR\C(\omega_r)} \mu^\RR_r(\widetilde C) $$
\end{thm}

\begin{rem}
Theorem \ref{corr 2} implies that the right hand side of last equality
does not depend on $\omega_r$ for smooth Del Pezzo toric surfaces
$Tor(\Delta)$.
 However, this is not true in general and one can easily check that the
sum of $r$-real multiplicities over all
tropical curves in $ \RR\C(\omega_r)$
in the case of $r>0$ does not have to be invariant if $Tor(\Delta)$ is singular
(see also \cite[Section 7.2]{Br8}). 
\end{rem}

\begin{exa}
If   $\omega_3=\{p_1,p_2,p_3,p_4,p_5\}$ is the configuration
depicted in Figure \ref{TCub}a, then images of all  parameterized
tropical curves of
genus 0 and
Newton polygon $\Delta_3$ in $\RR\C(\omega_3)$ are depicted in Figures
\ref{TCub}b, c, d, e, and f (compare with Table \ref{cubic}). Figure
\ref{TCub}e) represents the image of 2 distinct marked parameterized
tropical curves in $\RR\C(\omega_3)$, depending on the position of
marked points on the connected components of $\Im(\widetilde C)$.
Hence we verify that $W(\Delta_3,3)=2$.
\end{exa}

\begin{figure}[h]
\centering
\begin{tabular}{ccccccc}
\includegraphics[height=6cm, angle=0]{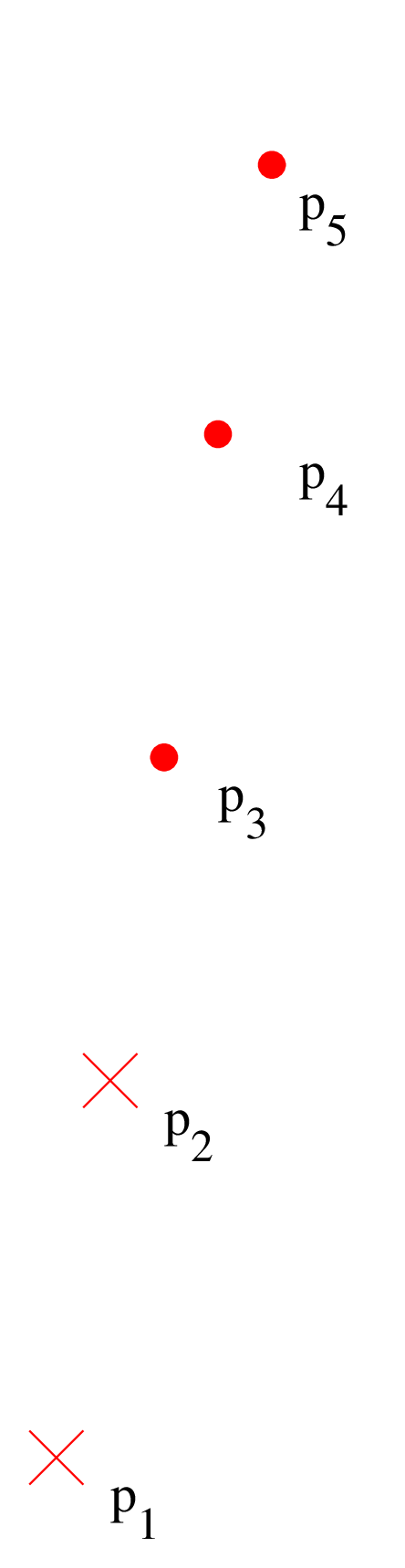}& \hspace{3ex} &
\includegraphics[height=6cm, angle=0]{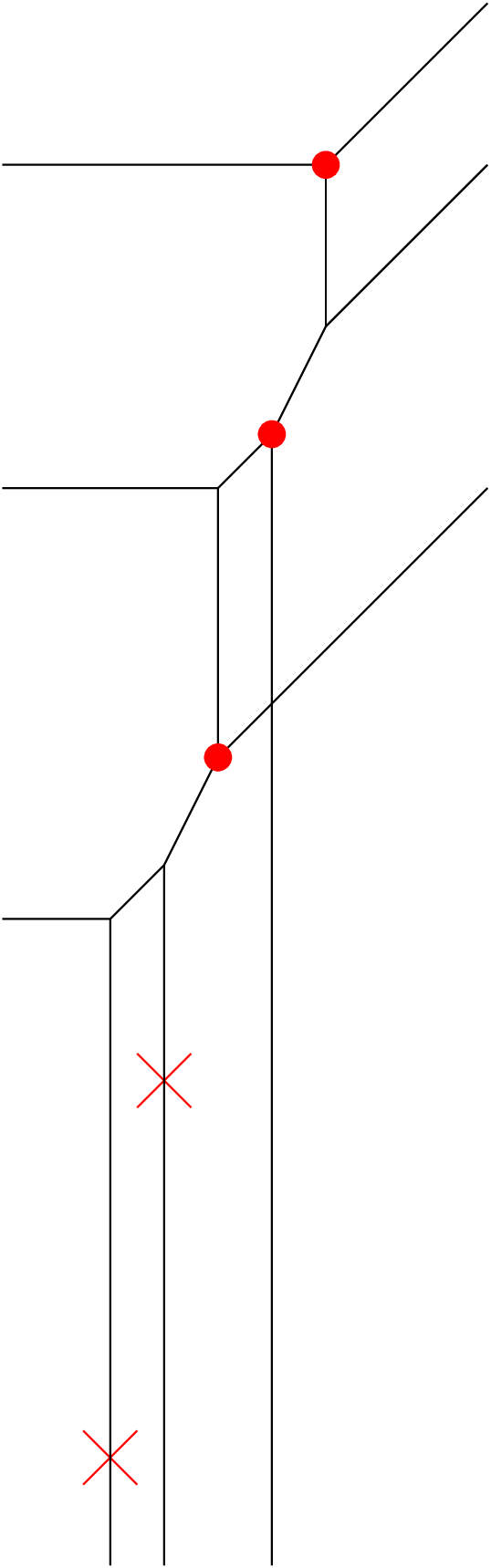}&
\includegraphics[height=6cm, angle=0]{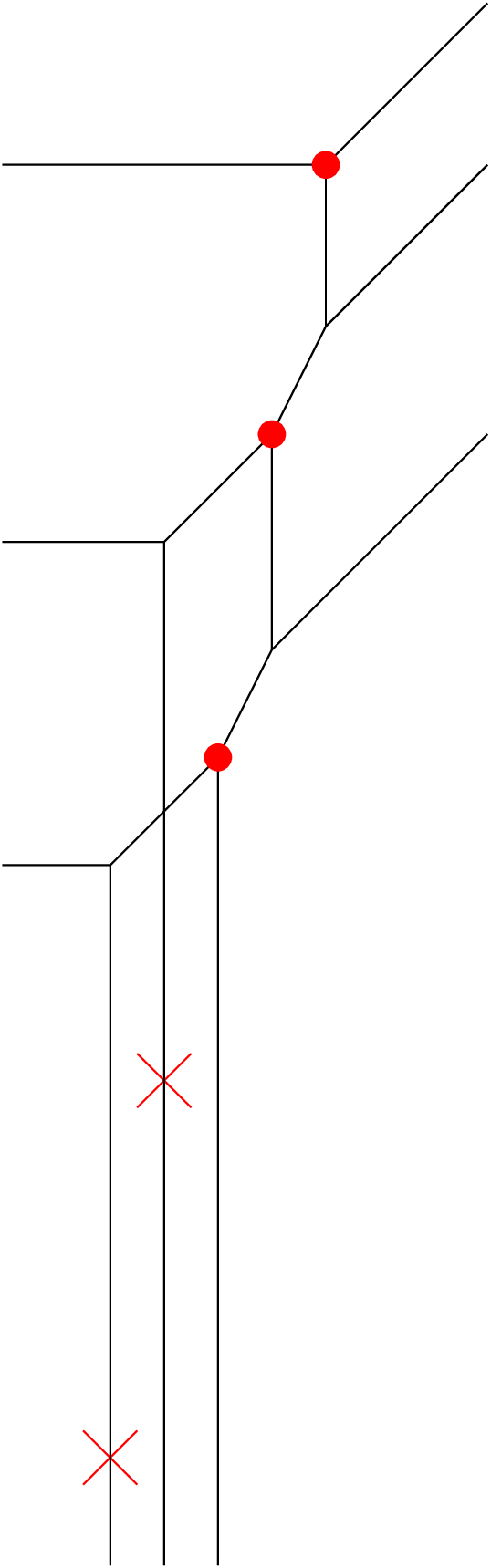}&
\includegraphics[height=6cm, angle=0]{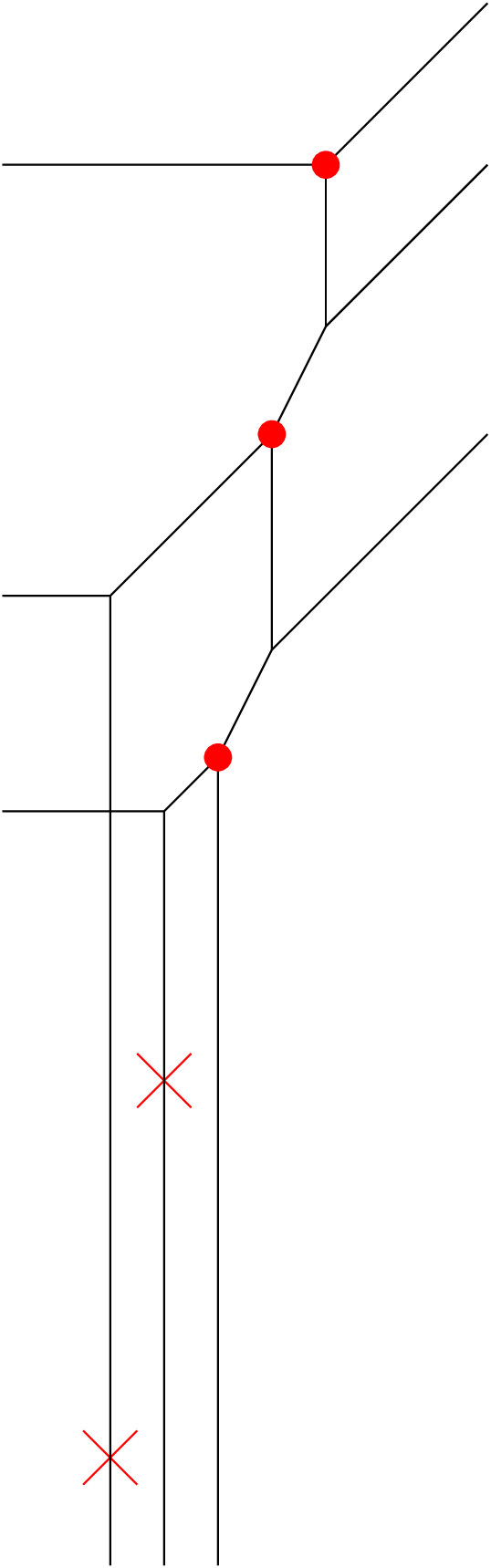}&
\includegraphics[height=6cm, angle=0]{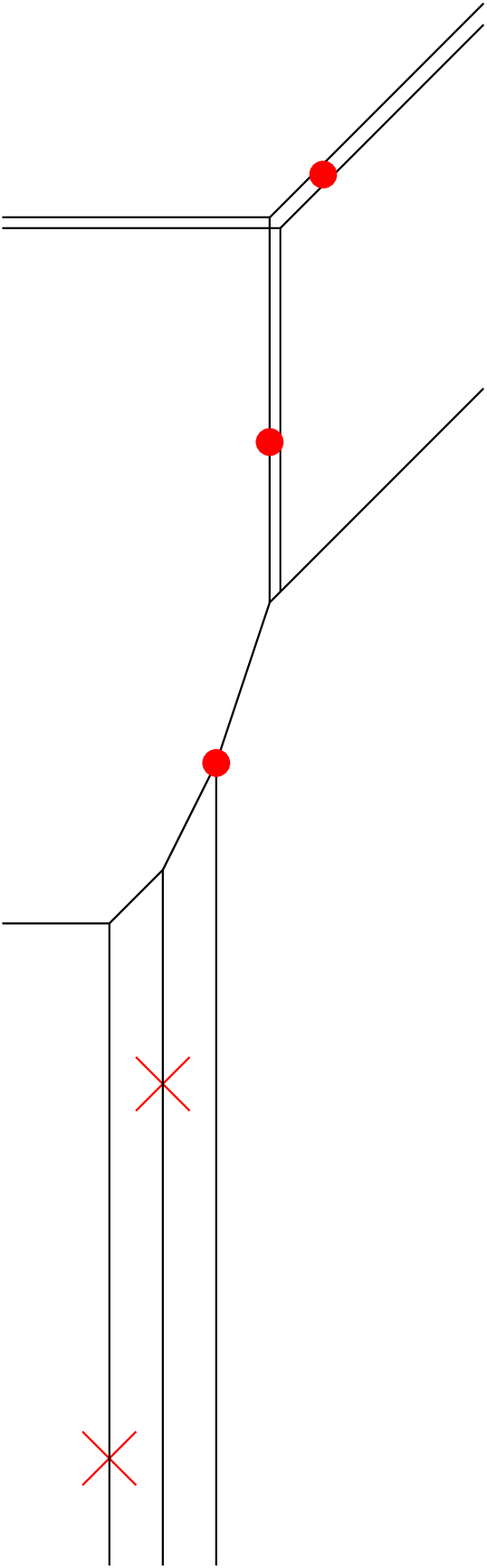}&
\includegraphics[height=6cm, angle=0]{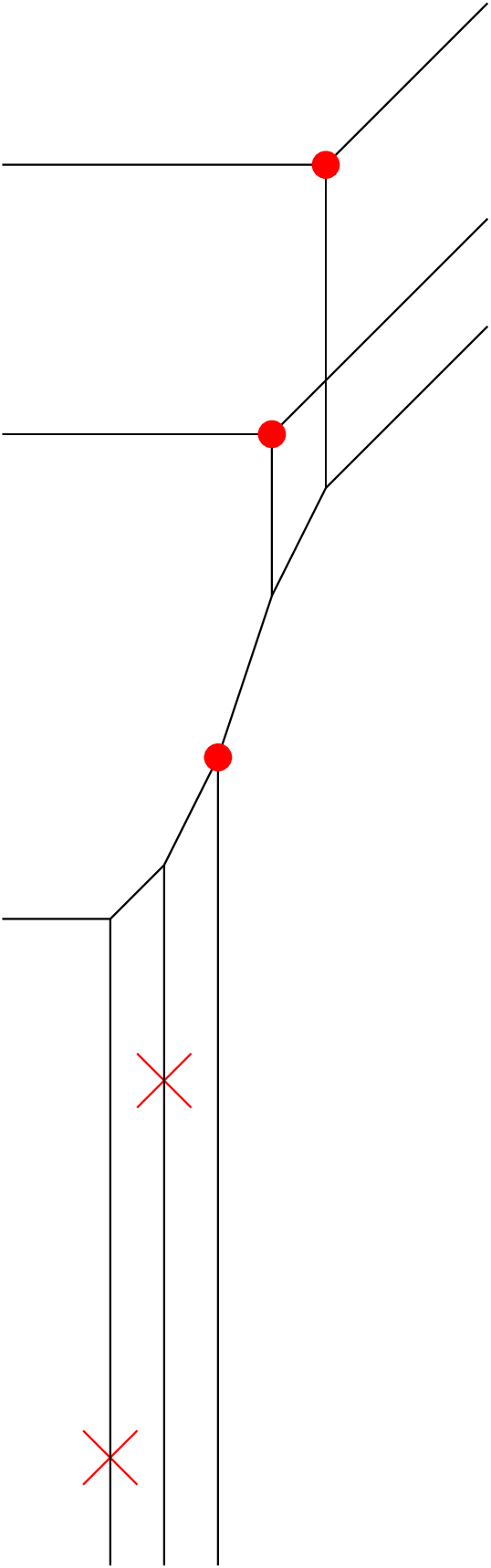}
\\
\\a) &&  b) $\mu^\RR_3=1$ & c)  $\mu^\RR_3=1$ & d) $\mu^\RR_3=1$&
 e) $\mu^\RR_3=-1$ & f)  $\mu^\RR_3=1$
\\ && & & & 2 choices
\end{tabular}
\caption{$W(\Delta_3,3)=2$}
\label{TCub}
\end{figure}

\section{Proof of Theorems \ref{NFD} and \ref{WFD}}\label{Floor trop}

Theorems \ref{NFD} and \ref{WFD} are obtained by applying Theorems
\ref{corr 1} and \ref{corr 2} to configurations $\omega$ which are
stretched  in
the direction $(0,1)$.

\subsection{Floors of a parameterized tropical curve}

As we fixed a preferred direction in $\RR^2$, it is natural to
distinguish between edges of parameterized tropical curves which are
mapped parallely to this direction from the others.

\begin{defi}
An elevator of a parameterized tropical curve $(C,f)$ is an edge $e$
of
$C$ with $u_{f,e}=\pm (0,1)$. The set of  elevators of $(C,f)$ is
denoted by $\E(f)$. If an elevator $e$ is not a leaf of $C$, then $e$
is said to be bounded.
A floor of a parameterized tropical curve $(C,f)$ is a connected
component of the topological closure of $C\setminus(\E (f)\cup
\text{End}(C))$.
\end{defi}

Naturally, a floor of a parameterized marked tropical curve is a floor
of the underlying parameterized tropical curve.

\begin{exa}
In Figure \ref{floor} are depicted some images of parameterized tropical
curves. Elevators are depicted in dotted lines.

\end{exa}

\begin{figure}[h]
\begin{center}
\begin{tabular}{ccccccc}
\includegraphics[height=2cm, angle=0]{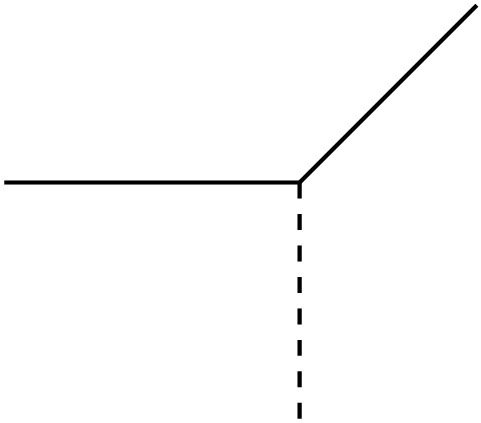}&
\hspace{3ex} &
\includegraphics[height=3cm, angle=0]{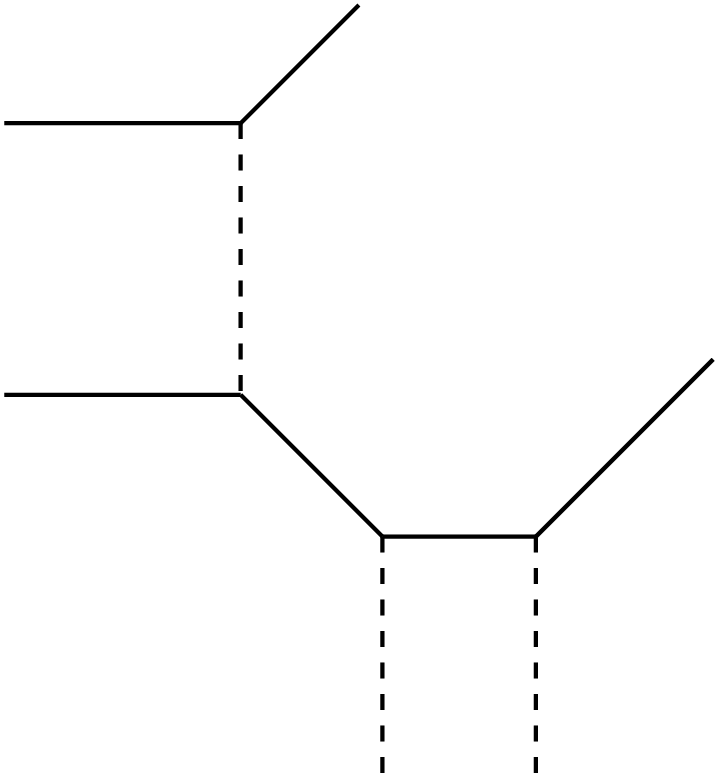}&
\hspace{3ex} &
\includegraphics[height=3cm, angle=0]{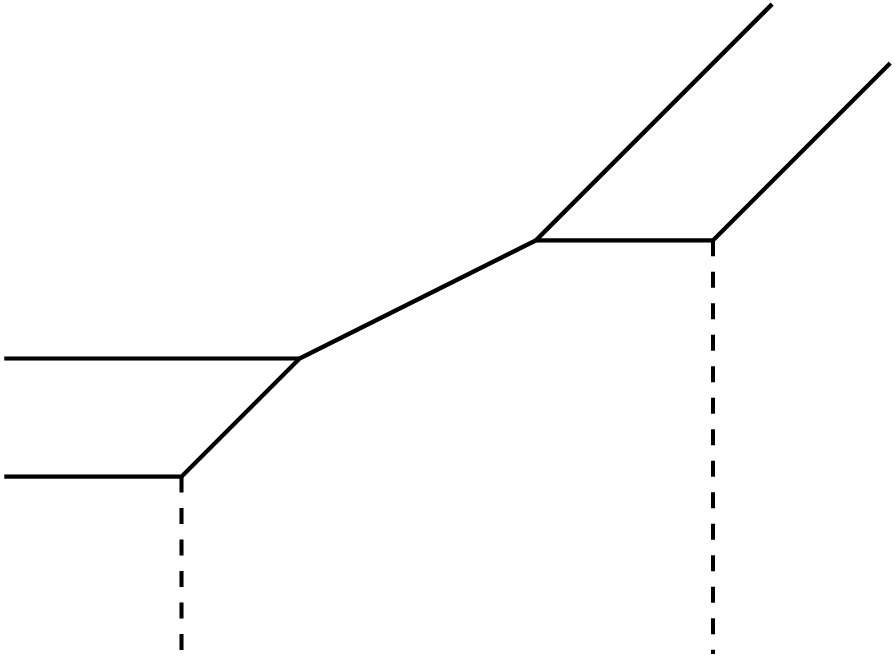}&
\hspace{3ex} &
\includegraphics[height=4cm, angle=0]{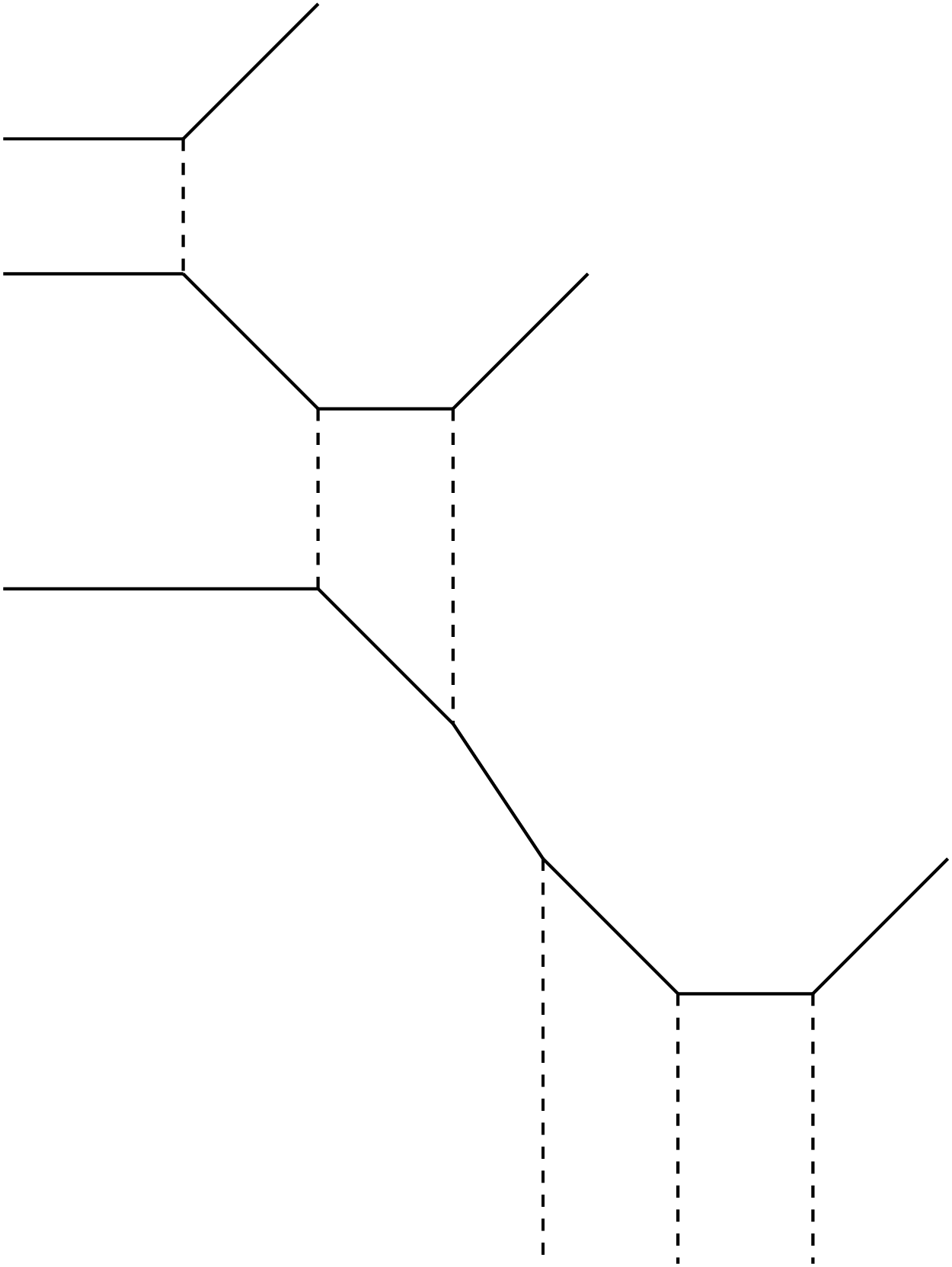}
\\
\\a) One floor && b) Two floors && c) One floor &&d) Three floors

\end{tabular}
\end{center}
\caption{Floors of tropical curves}
\label{floor}
\end{figure}

Let us fix a
$h$-transverse polygon $\Delta$, and a non-negative integer number $g$.
Define
$s=Card(\partial \Delta\cap \ZZ^2) -1 +g$, and choose a generic
configuration $\omega$  of $s$ points in
$\RR^2$.
If moreover $g=0$, choose
 $r$ an non-negative integer such that $s-2r\ge 0$,  and
choose
a collection $\omega_r$
of $s-r$
points in $\RR^2$.

\begin{prop}\label{strip}
Let $I=[a;b]$ be a bounded interval of $\RR$. Then, if $\omega$
(resp. $\omega_r$) is a subset
of $I\times\RR$, then any vertex of any curve in $\C(\omega)$
(resp. $\RR\C(\omega_r)$)  is mapped
 to  $I\times\RR$.
\end{prop}
\begin{proof}
Suppose that there exists an element $(C, x_1, \ldots,
x_s, f)$ in $\C(\omega)$
or $\RR\C(\omega_r)$, and a vertex $v$ of $C$ such that
$f(v)=(x_v,y_v)$ with $x_v<a$. Choose  $v$ such that no vertex of $C$
is mapped by $f$ to the half-plane $\{(x,y)\ | \ x<x_v\}$.
Suppose that $v$ is a trivalent vertex
of $C$, and denote by $e_1$, $e_2$ and $e_3$  the three edges of $C$
adjacent to
$v$. For $1\le i\le 3$, choose the vector $u_{e_i}$ pointing away
from $v$ (see section \ref{defi
trop curve}). By assumption on $v$, this vertex is adjacent to a
leaf of $C$, for example $e_1$, and since $\Delta$ is $h$-transverse
we have $u_{f,e_1}=(-1,\alpha)$. Moreover, according to Propositions
\ref{finite 1}
and \ref{finite 2}, we have
 $w_{f,e_1}=1$. By the balancing condition, up to exchanging $e_2$
and $e_3$, we have $u_{f,e_2}=(-\beta,\gamma)$ with $\beta\ge 0$, and
$u_{f,e_3}=(\delta,\varepsilon)$ with $\delta>0$.   Moreover, as no
vertex of $C$ is mapped to the half-plane  $\{(x,y)\ | \ x<x_v\}$, the
edge $f(e_2)$ is a leaf of $C$ if $\beta> 0$.
Then, by translating
the vertex $f(v)$  (resp. and possibly $\phi(v)$) in the direction
$u_{f,e_3}$,
we construct a
1-parameter family of parameterized tropical curves in $\C(\omega)$
(resp. $\RR\C(\omega_r)$), as depicted in two examples in Figure \ref{1
  param}. This
contradicts  Propositions \ref{finite 1}
and \ref{finite 2}. If
 $v$ is a 4-valent vertex of $C$, then we construct analogously a
1-parameter family of parameterized tropical curves in
 $\RR\C(\omega_r)$.
Alternatively, the contradiction may be derived from 
\cite[Lemma 4.17]{Mik1}.
Hence, no vertex of $C$ is mapped by $f$ in the
half-plane $\{(x,y)\ | \ x<a\}$.

The case where there exists an element $(C, x_1, \ldots,
x_s, f)$ in $\C(\omega)$
or $\RR\C(\omega_r)$, and a vertex $v$ of $C$ such that
$f(v)=(x_v,y_v)$ with $x_v>b$ works analogously.
\end{proof}

\begin{figure}[h]
\begin{center}
\begin{tabular}{ccc}
\includegraphics[height=3cm, angle=0]{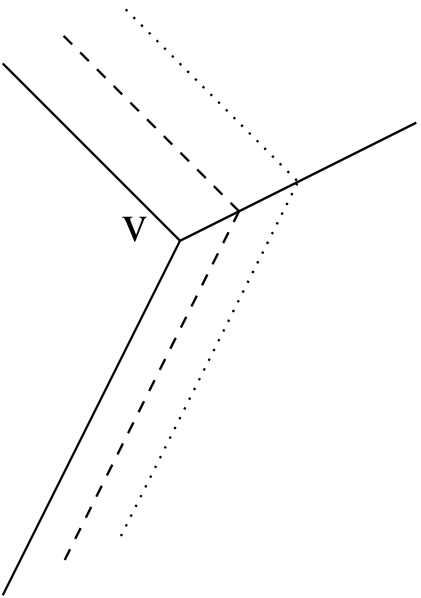}&
\hspace{5ex} &
\includegraphics[height=3cm, angle=0]{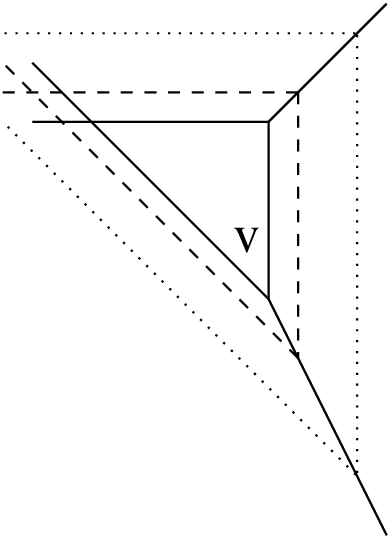}

\end{tabular}
\end{center}
\caption{1-parameter family of tropical curves}
\label{1 param}
\end{figure}

\begin{cor}\label{fd thm}
Let $I$ be a bounded interval of $\RR$. If $\omega$
(resp. $\omega_r$) is a subset
of $I\times\RR$ and if the points of $\omega$ (resp. $\omega_r$) are
far enough one from the others, then any floor of any curve in
$\C(\omega)$
(resp. $\RR\C(\omega_r)$)  can not contain more than one (resp. two)
distinct
marked point. If a floor of an element $\widetilde C$ in $\RR\C(\omega_r)$
contains two distinct marked points, then they are
contained in $\Im(\widetilde C)$.
\end{cor}
\begin{proof}
Let $(C, x_1, \ldots,
x_s, f)$ be an element of $\C(\omega)$
or $\RR\C(\omega_r)$ and choose a path $\gamma$ in $C\setminus
\mathcal E( C)$. The number of edges of $C$ is bounded
from above by a number which depends only on $\Delta$ and $g$, and
according to the tropical Bézout Theorem, absolute value of the
coordinates of the vector  $w_{f,e}u_{f,e}$ for any
edge $e$ of
$C$ is bounded from above by a number which depends only on $\Delta$.
According to Proposition \ref{strip}, all
vertices of $C$ are mapped by $f$ to the strip $I\times \RR$, so the
length (for the Euclidean metric in $\RR^2$) of $f(\gamma)$ is bounded
from above by a number $l_{max}(\Delta,g)$ which depends only on $\Delta$ and $g$. Hence,
if the distance between the points $p_i$ is greater than
$l_{max}(\Delta,g)$, two distinct marked points $x_i$ which are not mapped
to the same $p_j$ cannot be on the same
floor of $C$.
\end{proof}

For the remaining of this section, let us  fix a bounded interval $I$ of
$\RR$, a configuration $\omega=\{p_1,\ldots, p_s\}$, or possibly a
configuration
$\omega_r=\{p_1,\ldots, p_{s-r}\}$, such that the point $p_i$ is very
much
higher than the points $p_j$ if $j<i$. Here, \textit{very much higher}
means that we can apply Corollary \ref{fd thm}. Actually, we prove in
next corollary that any floor of any curve in $\C(\omega)$ or
$\RR\C(\omega_r)$ contains \textit{exactly} one marked point. More
precisely, we have the following statement.

\begin{cor}\label{all marked}
Let $\widetilde C$ be an element of  $\C(\omega)$ or
$\RR\C(\omega_r)$. Then, any floor of $\widetilde C$ contains exactly
one marked point. Moreover, the curve $\widetilde C$
 has exactly
$Card(d_l(\Delta))$ floors and $Card(d_l(\Delta))+g+d_-(\Delta)+
d_+(\Delta)-1$ elevators.
\end{cor}
\begin{proof}
Let us denote by $f_i$ (resp $b_i$, $\widetilde d_i$) the number of floors
(resp. bounded elevators, elevators) of $\widetilde C$
containing $i$ marked points. According to Corollary \ref{fd thm},
$f_i=0$ as soon as $i\ge 3$, and since the points $p_i$ are in general
position, we have $b_i=\widetilde d_i=0$ as soon as $i\ge 2$. What we
have to prove is that $f_0=f_2=b_0=\widetilde d_0=0$.
We have two
expressions for the number $s$ which gives us the equation
\begin{equation}\label{e1}
f_1+2f_2+\widetilde d_1 = d_+(\Delta)+d_-(\Delta)+ 2Card(d_l(\Delta)) -1 +g
\end{equation}
According to tropical Bézout Theorem and Corollary \ref{fd thm}, if a
floor of $C$ contains two marked points, then the intersection number of this
floor with a generic tropical line is at least 2. Hence we have

\begin{equation}\label{e4}
f_0 + f_1 + 2f_2\le Card(d_l(\Delta))
\end{equation}

According to Propositions \ref{finite 1} and \ref{finite 2},  we have
$d_++d_-$ leaves of $C$ which are elevators, thus
\begin{equation}\label{e3}
b_0+b_1=\widetilde d_0+\widetilde d_1 -(d_+(\Delta)+d_-(\Delta))
\end{equation}

An Euler characteristic computation shows us that
\begin{equation}\label{e2}
f_0+f_1 +f_2 - b_0 -b_1 \ge 1-g
\end{equation}

Combining Equations (\ref{e1}) with (\ref{e4}), then with Equation
(\ref{e3}), and finally with Equation
(\ref{e2}), we obtain

$$f_1 + f_2\ge Card(d_l(\Delta))$$
which is compatible with Equation (\ref{e4}) if an only if
$f_0=f_2=0$. Moreover, in this case
 inequalities
(\ref{e2}) and (\ref{e4}) are actually equalities, which  implies
 $b_0=\widetilde d_0=0$.
\end{proof}

\subsection{From tropical curves to floor diagrams}

To a parameterized tropical curve $(C,f)$, we  associate the following
oriented weighted graph, denoted by $\F(C,f)$~:  vertices of $\F(C,f)$
correspond to
floors of
$(C,f)$, and edges of $\F(C,f)$ correspond to elevators of
$(C,f)$. Edges of
$\F(C,f)$ inherit a natural weight from weight of $(C,f)$. Moreover,
$\RR$ is naturally oriented, and edges of
$\F(C,f)$ inherit this orientation, since they are all
parallel to the
coordinate axis
$\{0\}\times\RR$.
Note that we do \textit{not}  consider the graph  $\F(C,f)$ as  a
metric graph and that
some leaves are non-compact.

\begin{exa}
The graphs corresponding to parameterized tropical curves depicted in
 Figure \ref{floor} are depicted in Figure \ref{floor2}. Floors are
 depicted by ellipses, and elevators by segments. As all elevators have
 weight 1, we do not precise them on the picture. Orientation is
 implicitly from down to up.
\end{exa}
\begin{figure}[h]
\begin{center}
\begin{tabular}{ccccccc}
\includegraphics[height=1.3cm, angle=0]{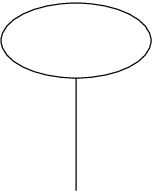}&
\hspace{3ex} &
\includegraphics[height=2cm, angle=0]{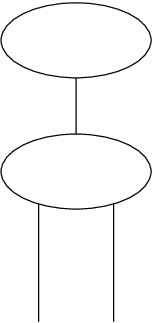}&
\hspace{3ex} &
\includegraphics[height=1.3cm, angle=0]{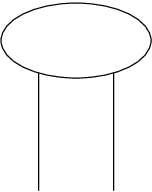}&
\hspace{3ex} &
\includegraphics[height=3cm, angle=0]{Figures/FD30.eps}
\\
\\a)  && b)  && c)  &&d)

\end{tabular}
\end{center}
\caption{Graphs associated to tropical curves}
\label{floor2}
\end{figure}

Let $\widetilde C$ be a parameterized tropical curve in $\C(\omega)$ or
$\RR\C(\omega_r)$. Since $\widetilde C$ has exactly
$Card(d_l(\Delta))$ floors, any floor $ \varepsilon$ of $\widetilde C$ has a
unique leaf $e$ with $u_{f,e}=(-1,-\alpha)$ where $u_{f,e}$ points to
infinity. Hence   the following map is well defined
$$\begin{array}{cccc}
\theta : & \text{Vert}(\F(\widetilde C)) &\longrightarrow & \ZZ
\\ & \varepsilon &\longmapsto & \alpha
\end{array} $$
The following lemma follows directly from Corollary \ref{all marked}
and Definition \ref{def fd}
of a floor
diagram.

\begin{lemma}
The graph $\F(\widetilde C)$ equipped with the map $\theta$ is a floor
diagram
of genus $g$ and Newton polygon $\Delta$.
\end{lemma}

Let us denote by $\D(\widetilde C)$ this floor diagram. Finally we
associate  to
 a parameterized tropical curve with $n$
marked points
 $\widetilde C=(C,x_1,\ldots,x_n,f)$ in $\C(\omega)$ or
$\RR\C(\omega_r)$ a marking $m$ of the floor
diagram $\D(\widetilde C)$. The
natural idea is to map the points $i$ to the floor or elevator of
$C$ containing $x_i$. However, it can happen if $\widetilde C$ is in
$\RR\C(\omega_r)$ that $x_i=x_{i+1}$ is a vertex $v$ of $C$. In this case,
according to Proposition \ref{finite 2} and Corollary \ref{all
  marked}, $v$ is on a floor $\varepsilon$  and is adjacent to an
elevator $e$ of $\D(\widetilde C)$. If $u_{f,e}=(0,1)$ points away from $v$
(resp. to $v$),
then  we define $m(i)=\varepsilon$ and $m(i+1)\in e$
(resp. $m(i+1)=\varepsilon$ and $m(i)\in e$). If $x_{i}$ is not a
vertex of $C$, then we define $m(i)$ as the floor or a point on the
edge of $C$ which contains $x_i$.

The map $m: \{1,\ldots , Card(\partial\Delta\cap \ZZ) -1+g\} \to
\D(\widetilde C)$ is
clearly an increasing
map, hence  it is a marking of the floor diagram $\D(\widetilde C)$.
In other
words, we have a map $\Phi : \widetilde C \mapsto (\D(\widetilde
C),m) $ from the set $\C(\omega)$ (resp.
$\RR\C(\omega_r)$) to the set of marked floor diagrams (resp. $r$-real
marked floor diagrams with non-null $r$-real multiplicity) of genus $g$
and Newton polygon $\Delta$.

\begin{exa}
All marked floor diagrams with a non-null complex multiplicity
(resp. $3$-real multiplicity)
 in Table \ref{cubic} correspond exactly to parameterized
tropical curves
whose image in $\RR^2$ are depicted in Figure \ref{Tcub0} (resp. \ref{TCub}).
\end{exa}

Theorems \ref{NFD} and \ref{WFD} are now a corollary of the next
proposition.
\begin{prop}\label{fd bij}
The map $\Phi$ is a bijection. Moreover, for any
element $\widetilde C$ in $\C(\omega)$ (resp.  $\RR\C(\omega_r)$), one has
$\mu^\CC(\widetilde C)=\mu^\CC(\Phi(\widetilde C))$ (resp.
$\mu^\RR_r(\widetilde C)=\mu^\RR_r(\Phi(\widetilde C))$).
\end{prop}
\begin{proof}
The fact that $\Phi$ is a bijection is
clear when $Card(d_l(\Delta))=1$. Hence the map $\Phi$ is always a bijection since
  an element of $\C(\omega)$ (resp.  $\RR\C(\omega_r)$) is
obtained by gluing, along elevators, tropical curves with a single
floor which are uniquely determined by the points $p_i$ they pass
through.

Let $\widetilde C=(C,x_1,\ldots, x_s,f)$ be an element of
$\C(\omega)$, and $v$ a vertex of
$C$. According to Corollary \ref{fd thm} and Corollary \ref{all marked}, $v$ is adjacent to an
elevator of weight $w$ and to an edge $e$ on a floor with
$u_{f,e}=(\pm 1,\alpha)$ and $w_{f,e}=1$. Hence,
$\mu^\CC(v,f)=w$. Since any leaf of $C$
 is of weight 1, it follows that $\mu^\CC(\widetilde C)$ is the product
of the square of the multiplicity of all elevators of $C$, that is
equal to $\mu^\CC(\Phi(\widetilde C))$.

 Let $\widetilde C=(C,x_1,\ldots, x_s,f,\phi)$ be an element of
 $\RR\C(\omega_r)$. The same
 argument as before shows  that $\mu^\RR_r(\widetilde C)$
and  $\mu^\RR_r(\Phi(\widetilde C))$ have equal absolute values.
It remains us to prove that both signs coincide, and the only
thing to check is that the number $o_r^\RR$ is even.  If $v$ is a
4-valent vertex of $C$ adjacent to an edge $e$ in $\Im(\widetilde C)$, then
$\mu^\RR(v)=w_{f,e}$. If $v$ is a
3-valent vertex in $\Re(\widetilde C)$ adjacent to an elevator $e$, then
$\mu^\CC(v)=w_{f,e}$. So if  $\mu^\CC(v)=3 \ mod \ 4$, then $e$ is bounded
and the other vertex $v'$ adjacent to $e$ satisfy also $\mu^\CC(v)=3
\ mod \ 4$. Hence the number $o_r^\RR$ is even as announced. 
\end{proof}

\section{Some applications}\label{application}

Here we use floor diagrams to confirm some results in classical
enumerative geometry.

\subsection{Degree of the discriminant hypersurface of the space of
  plane curves}

\begin{figure}[h]
\begin{center}
\begin{tabular}{c}
\includegraphics[height=5cm, angle=0]{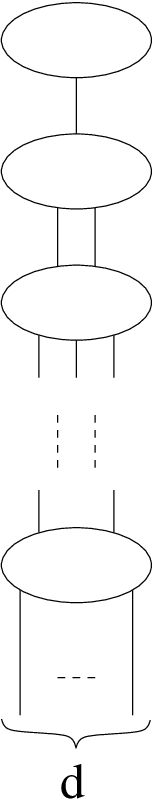}

\end{tabular}
\end{center}
\caption{Unique floor diagram of
  maximal genus and Newton polygon $\Delta_d$}
\label{max genus}
\end{figure}

\begin{prop}
For any $d\ge 3$, one has
$$N(\Delta_d,\frac{(d-1)(d-2)}{2}-1)= 3(d-1)^2 $$
\end{prop}
\begin{proof}
We see easily that the unique floor diagram $\D_{max}$ of genus
$\frac{(d-1)(d-2)}{2}$ and Newton polygon
$\Delta_d$  is the one depicted in
Figure \ref{max genus}. Moreover, all  floor diagrams of genus
$\frac{(d-1)(d-2)}{2}-1$ and Newton polygon
$\Delta_{d}$ are obtained by
decreasing the genus of $\D_{max}$ via one of the 2 moves depicted in
Figure \ref{decrease genus}. There are $i-1$ different markings of
the floor diagram obtained via the move of Figure \ref{decrease
  genus}a, and  $2i+1$ different markings of
the floor diagram obtained via the move of Figure \ref{decrease
  genus}b. Then we get
$$ \begin{array}{lll}
N(d,\frac{(d-1)(d-2)}{2}-1)&=& \sum_{i=2}^{d-1}4(i-1)
+\sum_{i=2}^{d}(2i-1)
\\
\\&=& 3(d-1)^2
\end{array}$$
\end{proof}

\begin{figure}[h]
\begin{center}
\begin{tabular}{cccc}
\includegraphics[height=2cm, angle=0]{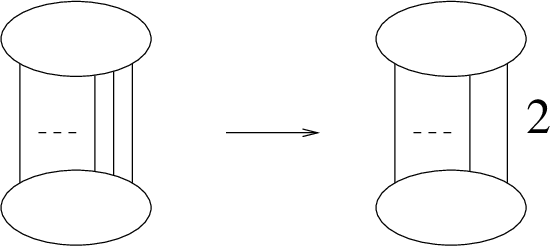}&
\hspace{10ex}  &
\includegraphics[height=1.5cm, angle=0]{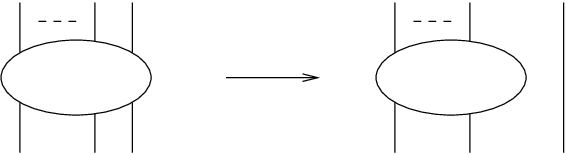}
\\ a) $i$ edges $\to$ $i-1$ edges &&b)  $i$ outgoing edges $\to$
$i-1$ outgoing edges
\\ $\mu^\CC=4$ && $\mu^\CC=1$
\end{tabular}
\end{center}
\caption{Decrease by 1 the genus of the floor diagram  of maximal genus}
\label{decrease genus}
\end{figure}

\subsection{Asymptotic of Welschinger invariants}

In \cite{Mik1}, a combinatorial algorithm in terms of \textit{lattice
  paths} has been
given to enumerate complex and real curves in toric surfaces. The idea
is that when we consider (the right number of) points which are
sufficiently far one from the other but on the same
line $L$ with irrational slope, then all tropical curves passing through
these points can be recovered inductively.
Hence, if $L$ is the line with equation $x+\varepsilon y$ with $y$ a very
small irrational number, then lattice paths and floor diagrams are two
ways to encode
the same tropical curves. However, in our opinion, floor diagrams are
much easier to deal with. In particular, one does not have to consider
reducible curves using floor diagrams.

As an example, we give a floor diagram proof of the following
theorem that was initially proved with the help of the lattice paths.

\begin{thm}[Itenberg, Kharlamov, Shustin \cite{IKS1} \cite{IKS2}]
The sequence $(W(\Delta_d,0))_{d\ge 1}$ satisfies the following properties :

\begin{itemize}
\item[$\bullet$] it is a sequence of positive numbers,

\item[$\bullet$] it is an increasing sequence, and strictly increasing
  starting from $d=2$,

\item[$\bullet$] one has $\ln  W(\Delta_d,0) \sim \ln N(\Delta_d,0)
  \sim 3d \ln  d$ when $d$ goes to infinity.

\end{itemize}
\end{thm}
\begin{proof}
As we have $\mu^\RR_0=1$ for any floor
diagram, the numbers $W(\Delta_d,0)$ are all non-negative. Moreover,
we have $W(\Delta_1,0)=1$ so
the
positivity of these numbers will follow from the increasingness of the
sequence  $(W(\Delta_d,0))_{d\ge 1}$.

Let $(\mathcal D_0,m_0)$ be a marked floor diagram of genus 0 and
Newton polygon $\Delta_d$.
For convenience we use marking $m_0:\{4,\ldots,3d+2\}\to D_0$
(instead of the ``usual" marking $\{1,\ldots,3d-1\}\to\mathcal D_0$).
Note that the point $4$ has to be mapped to an edge
in $\text{Edge}^{-\infty}(\D_0)$. Out of
 $\mathcal D_0$, we can construct a new marked floor diagram $\D$ of
genus 0 and Newton polygon $\Delta_{d+1}$ as indicated in
 Figure \ref{FDW}a. Both real multiplicities $\mu^\RR_0(\D_0)$ and
 $\mu^\RR_0(\D)$ are the same, and two distinct marked floor diagrams
 $\mathcal D_0$  and $\mathcal D_0'$  give rise to two distinct marked
 floor diagrams
 $\mathcal D$  and $\mathcal D'$. Hence, we have $W(\Delta_{d+1},0)\ge
 W(\Delta_{d},0)$ for all $d\ge 1$. Moreover, if $d\ge 2$ then there
 exist marked  floor diagrams with Newton polygon $\Delta_{d+1}$ which are not
 obtained out of a marked floor diagrams with Newton polygon
 $\Delta_{d}$ as described above. An example is given in Figure
\ref{FDW}b, hence $W(\Delta_{d+1},0)>W(\Delta_d,0)$ if $d\ge 2$.

\begin{figure}[h]
\begin{center}
\begin{tabular}{ccc}
\includegraphics[height=4cm, angle=0]{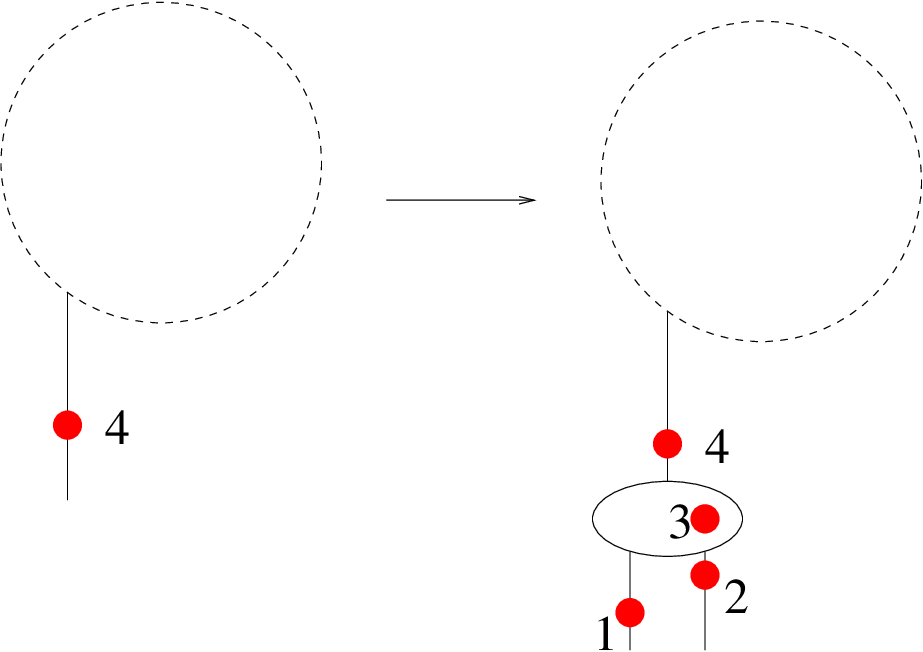}&
\hspace{10ex}  &
\includegraphics[height=5cm, angle=0]{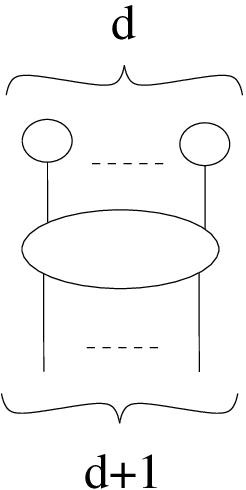}
\\
\\ a) From $\Delta_d$ to $\Delta_{d+1}$ && b) Not obtained from  $\Delta_d$

\end{tabular}
\end{center}
\caption{The numbers $W(\Delta_d,0)$ are increasing}
\label{FDW}
\end{figure}

We study now the logarithmic asymptotic of the sequence
$(W(\Delta_d,0))_{d\ge 1}$. For simplicity, we only
  treat the case of the subsequence $(\log(W(\Delta_{2^{k}},0))_{k\ge 1}$. The
  general case does not require additional idea, but is more
  technical.
  The interested reader is referred to
  \cite[Appendice]{Br10} for a complete proof in the floor diagrams setting.
Let $(\mathcal D_d)_{d\ge 1}$  be the
sequence of floor diagrams
constructed inductively in the following way : $\mathcal D_1$ is the
floor diagram
with Newton polygon $\Delta_1$, and $\mathcal D_d$ is obtained out of
 $\mathcal D_{d-1}$ by gluing to each edge in $\text{Edge}^{-\infty}(\D_{d-1})$ the
piece depicted in Figure \ref{FDWa}a. Floor diagrams
$\mathcal D_1$, $\mathcal D_2$, $\mathcal D_3$, and $\mathcal D_4$ are
depicted in  Figures  \ref{FDWa}b, c, d et e. The floor diagram $\D_d$
is of degree  $2^{d-1}$ and we have $\mu^\RR_0( \D_d)=1$. If
 $\nu(\mathcal D_d)$ denotes the number of distinct markings of $\D_d$,
then we have

$$ \begin{array}{llll}
\forall d \ge 2 \hspace{4ex}& \nu(\mathcal D_d) &=& \frac{\nu(\mathcal
  D_{d-1})^2}{2}C_{3\times 2^{d-1}-4}^{3\times 2^{d-2}-2}
\\
\\& &=& \frac{(3\times 2^{d-1}-4)!}{2^{2^{d-1}-1}}\prod_{i=2}^{d}\frac{1}{\left((3\times 2^{d-i}-2)(3\times 2^{d-i}-3)\right)^{2^{i-1}}}

\end{array}$$

Hence we get

$$\frac{(3\times 2^{d-1}-4)!}{2^{2^d }\prod_{i=1}^{d}(3\times
  2^{d-i})^{2^{i}}} \le \nu(\mathcal D_d) \le  (3\times 2^{d-1}-4)! $$

The Stirling Formula implies that
 $\ln d!\sim d\ln d$, and we see easily that both right and left
hand side of the inequality have the same logarithmic asymptotic,
namely
 $ 3\times
2^{d-1}\ln (2^{d-1})$. As we have $\ln \nu(\mathcal D_d)\le \ln W(\Delta_{2^{d-1}},0)
\le \ln N(\Delta_{2^{d-1}},0)$, the result follows from
the equivalence
$\ln(N(\Delta_d,0))\sim 3d\ln
d$ proved in  \cite{DiFrIt}.
\end{proof}

\begin{figure}[h]
\begin{center}
\begin{tabular}{ccccccccc}
\includegraphics[height=1.5cm, angle=0]{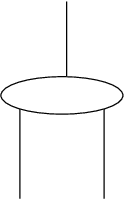}&
\hspace{2ex}  &
\includegraphics[height=1.3cm, angle=0]{Figures/FD1.eps}&
\hspace{2ex}  &
\includegraphics[height=2cm, angle=0]{Figures/FD2.eps}&
\hspace{2ex}  &
\includegraphics[height=3cm, angle=0]{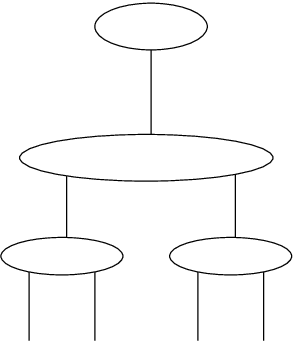}&
\hspace{2ex}  &
\includegraphics[height=4cm, angle=0]{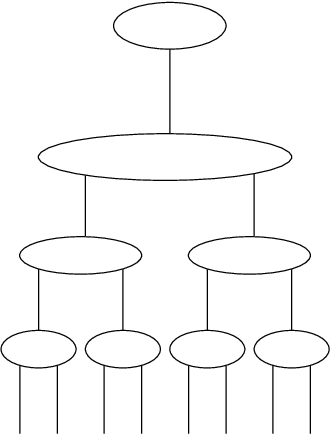}
\\
\\ a) &&b) $\mathcal D_1$ &&
c) $\mathcal D_2$&&d) $\mathcal D_3$&& e) $\mathcal D_4$
\end{tabular}
\end{center}
\caption{Asymptotic of the numbers $W(\Delta_d,0)$}
\label{FDWa}
\end{figure}

\subsection{Recursive formulas}

Floor diagrams allow one to write down easily recursive formulas in a
Caporaso-Harris style (see \cite{CapHar1}) for both complex and real
enumerative invariants. The recipe to extract such
formulas is explained in \cite{Br8} in the particular case of the
numbers $W(\Delta_d,r)$.

As an example we briefly outline here how to reconstruct Vakil's formula \cite{Vak2},
which relates some enumerative invariants of Hirzebruch surfaces.

The
Hirzebruch surface $\FF_n$ of degree $n$, with $n\ge 0$, is the
compactification of the line bundle over $B=\mathbb CP^1$ with first
Chern class $n$. If $\FF_n\supset F\approx \mathbb CP^1$ denotes the
compactification of a fiber, then the second homology group of $\FF_n$
is the free abelian group generated by $B$ and $F$.
In a suitable coordinate system, a  generic algebraic curve in $\FF_n$
of class $aB+
bF$, with $a,b\ge 0$, has the $h$-transverse Newton polygon
$\Delta_{n,a,b}$ with vertices
$(0,0)$, $(na +b,0)$, $(0,a)$, and $(b,a)$ (see \cite{Beau} for more
details about Hirzebruch surfaces).

Before stating the theorem, we need to introduce some notations. In
the following, $\alpha=(\alpha_1,\alpha_2,\ldots)$  denotes a sequence
of non-negative integers, and we set
$$|\alpha|=\sum_{i=1}^\infty \alpha_i, \ \ \ \ \ \
I\alpha = \sum_{i=1}^\infty i\alpha_i, \ \ \ \ \ \
I^\alpha = \prod_{i=1}^\infty i^{\alpha_i} $$

If $a$ and $b$ are two integer numbers, $\left(\begin{array}{c}
a\\ b\end{array} \right) $ denotes the binomial coefficient. If $a$
and $b_1,b_2,\ldots,b_k$ are
integer numbers then $\left(\begin{array}{c}
a
\\ b_1,\ldots, b_k
\end{array} \right) $ denotes the multinomial coefficient, i.e.
$$\left(\begin{array}{c}
a
\\ b_1,\ldots, b_k
\end{array} \right) = \prod_{i=1}^{k}\left(\begin{array}{c}
a -\sum_{j=1}^{i-1}b_j
\\ b_i
\end{array} \right)$$

\begin{thm}[Vakil, \cite{Vak2}]\label{vakil}
For any $n\ge 0$, any $g\ge 0$, and any $b\ge 1$, one has
$$
N(\Delta_{n,2,b},g) =  N(\Delta_{n+1,2,b-1},g) \text{\hspace{60ex}}$$
$$ \ \ \ \ \ \  \ \ \ \ \ \ \ \ \ +
\sum_{\begin{array}{c}I\beta \le n \\ |\beta|=g+1\end{array}}
\left(\begin{array}{c}2n+2b + g+2 \\ n-I\beta \end{array}  \right)
\left(\begin{array}{c}\beta_1+b \\ b \end{array}  \right)
\left(\begin{array}{c}|\beta|+b \\ \beta_1+b,\beta_2,\beta_3,\ldots  \end{array}  \right)
I^{2\beta}
$$
\end{thm}
\begin{proof}
We want to enumerate marked floor diagrams of genus $g$ and Newton
polygon $\Delta_{n,2,b}$. As these floor diagrams have only two
floors, our task is easy. Let $\D$ be such a marked floor
diagrams of genus $g$ and Newton
polygon $\Delta_{n,2,b}$. Then, the marking $m$ is defined on the set
$\{1,\ldots, s\}$ where $s=2(n+2) + 2b -1 +g$.

 Suppose that $m(s)$ is a floor of $\D$. These
  marked floor diagrams are easy to enumerate, their contribution to the
  number $N(\Delta_{n,2,b},g)$ is the second term on the
  right hand side of the equality.

 Suppose that $m(s)$ is on an edge $e$ in
  $\text{Edge}^{+\infty}(\D)$. Define a new floor diagram $\D'$ as
  follows : $\text{Vert}(\D')=\text{Vert}(\D')$,
  $\text{Edge}(\D')=(\text{Edge}(\D) \setminus \{e\})\cup
  \{e'\}$,
  where $e'$ is in $\text{Edge}^{-\infty}(\D)$ and is adjacent to the other floor than $e$.
  Define a marking $m'$ on $\D'$ as follows : $m'(i)=m(i-1)$ if $i\ge
  2$ and $m(1)\in e'$. Now, the marked floor diagram $(\D',m')$ is of
  genus $g$ and Newton polygon $\Delta_{n+1,2,b-1}$ (see
  Figure \ref{blow up down}a). Moreover, we
  obtain in this way a bijection between the set of marked floor
  diagrams  of genus $g$ and Newton
polygon $\Delta_{n,2,b}$ such that $m(s)\in\text{Edge}^{+\infty}(\D)$,
and marked floor diagrams of
  genus $g$ and Newton polygon $\Delta_{n+1,2,b-1}$. Hence, we get the
  first term of the right hand side of the equality, and the theorem is proved.
\end{proof}

\begin{figure}[h]
\begin{center}
\begin{tabular}{ccc}
\includegraphics[height=3.5cm, angle=0]{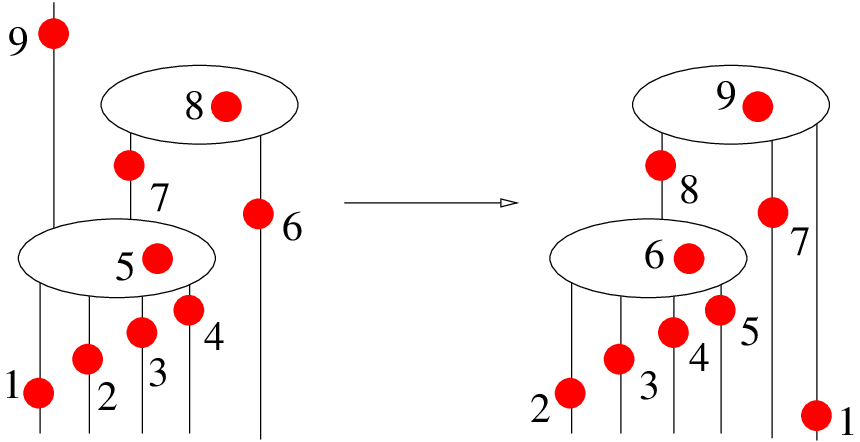}&
\hspace{3ex} &
\includegraphics[height=3.5cm, angle=0]{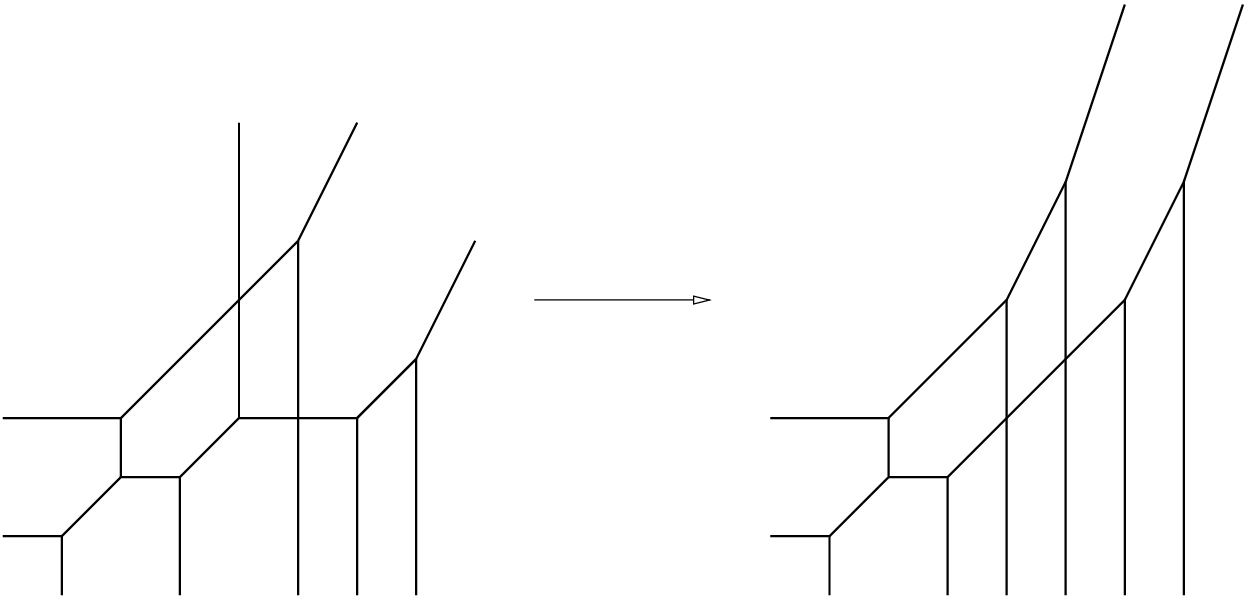}
\\
\\ a)&& b)
\end{tabular}
\end{center}
\caption{From $\FF_n$ to $\FF_{n+1}$}
\label{blow up down}
\end{figure}

\begin{rem}
Our proof of Theorem \ref{vakil} is a combinatorial game on
marked floor diagrams that can be obtained as the translation to
the floor diagram language of Vakil's original proof : take the highest point $p$ of
the configuration, and specialize it to the exceptional section
$E$. Then, either a curve $C$ we are counting breaks into 2 irreducible components,
which give the second term, or $C$ has now a prescribed point on
$E$. Blowing up this point and blowing down the strict transform of
the fiber, the curve is transformed to a curve in $\FF_{n+1}$ with a
prescribed point on $B$ (which is the image under the blow down of the
second intersection
point of $C$ with the fiber).

The effect of such a blow up and down in tropical geometry can be easily
seen, since intersection points with $E$ correspond to leaves
going
up, and intersection with $B$ correspond to leaves going
down. An example is given in Figure \ref{blow up down}b which
correspond to the operation on marked floor diagram depicted in Figure
\ref{blow up down}a.
\end{rem}

\section{Further computations}\label{further}

One can adapt the technics of this paper to compute other real and
complex enumerative
invariants of algebraic varieties. In addition to genus 0 Gromov-Witten
invariants  and Welschinger invariants of higher dimensional
spaces, as announced in \cite{Br7}, one can compute in this way
characteristic numbers of the projective plane (at least in genus 0
and 1), as well as Gromov-Witten and Welschinger invariants of the
blown up projective plane.
Details will appear soon.

\small
\def\rightmark{\em Bibliography}
\addcontentsline{toc}{section}{References}

\bibliographystyle{alpha}
\bibliography{../../Biblio.bib}

\end{document}